\documentclass{amsart}

\usepackage{graphicx} 

\usepackage{amsmath,amssymb,hyperref,mathabx,color,enumerate,chngpage}

\usepackage{mathrsfs}

\usepackage{lipsum}
\usepackage{amsfonts}
\usepackage{graphicx}
\usepackage{epstopdf}
\usepackage{algorithm}
\usepackage{algpseudocode}
\ifpdf
  \DeclareGraphicsExtensions{.eps,.pdf,.png,.jpg}
\else
  \DeclareGraphicsExtensions{.eps}
\fi

\newtheorem{definition}{Definition}[section]

\newtheorem{theorem}{Theorem}[section]
\newtheorem{lemma}[theorem]{Lemma}

\newcommand{\algn}[1]{\begin{align} #1 \end{align}}
\newcommand{\algns}[1]{\begin{align*} #1 \end{align*}}

\newcommand{\mc}[1]{\mathcal{#1}}
\newcommand{\LRp}[1]{\left( #1 \right)}

\newcommand{\dive}{\operatorname{div}}
\newcommand{\dt}{d_t}

\title[Optimal control of thermo-poroelasticity]{Numerical approximation of distributed optimal control problems governed by linear thermo-poroelasticity}
\author{Jeonghun J. Lee}
\address{Department of Mathematics, Baylor University}
\email{Jeonghun\_Lee@baylor.edu}

\begin{document}


\begin{abstract}
	In this paper, we propose and analyze a novel three-field symmetric formulation for optimal control problems governed by linear thermo-poroelasticity. The state variables are the solid displacement $\boldsymbol{u}$, the fluid pressure $p$, and the temperature $\theta$, and both the distributed fluid source $m_p$ and the distributed heat source $m_\theta$ serve as control variables. We prove well-posedness of the symmetric differentiated formulation, establish the existence and uniqueness of an optimal control pair $(\bar{m}_p,\bar{m}_\theta)$, and derive the first-order necessary optimality conditions via a coupled three-field adjoint system. For the fully discrete scheme based on dG(0) time discretization, we employ a variational discretization approach for both controls and derive a priori error estimates of order $\mathcal{O}(h^s + \Delta t)$ for the state, adjoint, and control errors. We also present manufactured-solution experiments that validate the fully discrete optimality system and illustrate the storage-degenerate regime covered by our analysis.
\end{abstract}
\keywords{Thermo-poroelasticity; Biot's consolidation model; optimal control; heat source control; fluid source control; finite element methods; variational discretization; a priori error estimates.}
\subjclass{65M12, 65M15, 65M60, 76S99}

\maketitle

\section{Introduction}
Mathematical models that describe how fluids move and how materials deform within porous structures are essential in many fields, such as geophysics, engineering, medicine and biological applications, such as in the development of the robust mechanical modelling of tissues and cells \cite{malandrino2019poroelasticity} and the research on the poroelastic component of the postseismic process \cite{mccormack2020modeling}. These models help us understand fluid flow in natural and artificial materials like cartilage, bones, tissue scaffolds, and even blood flow through organs like the brain, liver, and eyes. In many of these settings, mechanical deformation and fluid flow are also strongly coupled with heat transfer: temperature gradients drive additional fluid and solid motion through thermal expansion, and conversely deformation and fluid flux alter the temperature field. Such thermally coupled, or \emph{thermo-poroelastic}, effects are central to geothermal energy extraction, the disposal of heat-generating nuclear waste in deep geological repositories, thermally enhanced oil recovery, and the thermal regulation of biological tissue, so a rigorous mathematical and numerical treatment of the coupled three-field system is of independent interest beyond the isothermal setting. Because of their wide range of applications, flow problems in porous media, and their thermally coupled extensions, have become a major topic of study in mathematics. Many researchers have focused on understanding the stability of these models, analyzing how sensitive they are to changes, and developing reliable numerical methods to simulate them.

The foundational model in this area is Biot's consolidation model \cite{biot1941general}, which describes the interaction between interstitial fluid flowing through deformable porous media and has many applications, such as groundwater \cite{Kim1999549}, biological tissues \cite{Swan200325}, carbon sequestration \cite{WANGEN2016486}, and materials science \cite{MR2273503}. It plays a central role in simulating organ behaviour in medical applications and predicting how porous rocks respond in geophysical studies. The thermo-poroelastic system studied in this paper extends Biot's model by additionally coupling the temperature field to the displacement and pressure, reflecting the thermo-hydro-mechanical interactions present in the applications above. In recent years, finite element methods have been widely developed and studied as powerful tools for solving Biot's model, and its thermally coupled extensions, accurately and efficiently.

\subsection{The thermo-poroelastic consolidation model}\label{sec:Biot-model}
In this paper, we focus on quasi-static thermo-poroelastic consolidation in the spirit of \cite{MR3803860}, extending the three-field poroelastic formulation therein by the addition of a coupled temperature field. Suppose that $\Omega \subset \mathbb{R}^d$, $d =2,3$, is a bounded domain with Lipschitz polygonal/polyhedral boundary $\partial \Omega$. We assume three independent partitions of $\partial \Omega$:
\begin{align*}
	\partial \Omega &= \Gamma_p \cup \Gamma_{f}, & \Gamma_p \cap \Gamma_f &= \emptyset, & |\Gamma_p| &> 0, \\
	\partial \Omega &= \Gamma_d \cup \Gamma_{t}, & \Gamma_d \cap \Gamma_t &= \emptyset, & |\Gamma_d|, |\Gamma_t| &> 0, \\
	\partial \Omega &= \Gamma_\theta \cup \Gamma_{\theta,f}, & \Gamma_\theta \cap \Gamma_{\theta,f} &= \emptyset, & |\Gamma_\theta| &> 0.
\end{align*}
governing the fluid, mechanical, and thermal boundary conditions, respectively. Here $\Gamma_t$ is used for traction boundary conditions and it is not related to the time variable $t$. We seek the displacement $\boldsymbol{u}(t): \Omega \rightarrow \mathbb{R}^{d}$, the fluid pressure $p(t): \Omega \rightarrow \mathbb{R}$, and the temperature $\theta(t): \Omega \rightarrow \mathbb{R}$ satisfying
\begin{subequations}
    \label{sys-eqs}
\begin{align}
	\label{sys1} - \text{div} (\boldsymbol{\mathcal{C}} \epsilon(\boldsymbol{u})) + \alpha_p \nabla p + \alpha_\theta \nabla \theta &= \tilde{\boldsymbol{f}} && \text{in} \ (0,T] \times \Omega,\\
	\label{sys2} s_{pp} \dot{p} + \alpha_p \text{div}(\dot{\boldsymbol{u}}) + s_{p\theta}\dot{\theta} -  \text{div} (\kappa_p \nabla p) &= -m_p && \text{in} \ (0,T] \times \Omega,\\
	\label{sys3} s_{\theta\theta} \dot{\theta} + \alpha_\theta \text{div}(\dot{\boldsymbol{u}}) + s_{\theta p}\dot{p} - \text{div}(\kappa_\theta \nabla \theta) &= -m_\theta && \text{in} \ (0,T] \times \Omega,
\end{align}
\end{subequations}
where $\mathcal{C}$ is the elastic stiffness tensor, $\epsilon(\boldsymbol{u}) = \frac{1}{2}\left(\nabla \boldsymbol{u} + \left(\nabla \boldsymbol{u}\right)^{T}\right)$ is the strain tensor, $\alpha_p >0$ is the Biot--Willis coefficient and $\alpha_\theta > 0$ is the thermal expansion coupling coefficient, $\tilde{\boldsymbol{f}}(t): \Omega \rightarrow \mathbb{R}^{d}$ is the body force, $s_{pp} \ge 0$ is the specific storage coefficient, $s_{\theta\theta} \ge 0$ is the volumetric heat capacity, $s_{p\theta} = s_{\theta p}$ are the symmetric thermo-poroelastic coupling coefficients, $\kappa_p$ is the permeability tensor and $\kappa_\theta$ is the thermal conductivity tensor (both $d\times d$ symmetric positive definite with constant entries), $m_p(t):\Omega\to\mathbb{R}$ is the distributed fluid source (control), and $m_\theta(t):\Omega\to\mathbb{R}$ is the distributed heat source (control). The dot denotes the time derivative. 
For isotropic elastic porous media, $\boldsymbol{\mathcal{C}} \epsilon(\boldsymbol{u}) = 2\mu  \epsilon(\boldsymbol{u}) + \lambda \text{div}(\boldsymbol{u}) \mathbb{I}$,
where $\mathbb{I}\in \mathbb{R}^{d\times d}$ is the identity matrix, $\lambda$ and $\mu$ are Lam\'{e} constants related to the Young's elasticity modulus $E$ and the Poisson’s ratio $0<\nu<1/2$ as
\begin{align*}
	\lambda = \frac{\nu E}{(1-2\nu)(1+\nu)},\qquad \qquad \mu = \frac{E}{2(1+\nu)}.
\end{align*}
The boundary conditions are given by
\begin{align*}
	p(t) &=0 \quad \text{on} \ \Gamma_p, & -\kappa_p \nabla p(t) \cdot \boldsymbol{n}  &= 0 \quad \text{on} \ \Gamma_{f},
    \\
	\boldsymbol{u}(t) &=0 \quad \text{on} \ \Gamma_d, & \ \ (\boldsymbol{\mathcal{C}} \epsilon(\boldsymbol{u}(t)) - \alpha_p p(t) \mathbb{I} - \alpha_\theta \theta(t) \mathbb{I}) \boldsymbol{n} &= \boldsymbol{0} \quad \text{on} \ \Gamma_t,\\
	\theta(t) &= 0 \quad \text{on} \ \Gamma_\theta, & -\kappa_\theta \nabla \theta(t) \cdot \boldsymbol{n} &= 0 \quad \text{on} \ \Gamma_{\theta,f},
\end{align*}
for all $t \in (0,T]$, where $\boldsymbol{n}$ is the outward unit normal on $\partial \Omega$. In the rest of this paper we assume homogeneous boundary conditions for simplicity of presentation.

In this paper we consider a reformulation of \eqref{sys-eqs} by taking the time derivative of the mechanical equation \eqref{sys1}:
\begin{subequations}
	\label{eq:new-strong-eqs}
\begin{align}
	\label{eq:new-strong-eq1} - \text{div} (\boldsymbol{\mathcal{C}} \epsilon(\dot{\boldsymbol{u}})) + \alpha_p \nabla \dot{p} + \alpha_\theta \nabla \dot{\theta} &= {\boldsymbol{f}} && \text{in} \ (0,T] \times \Omega,\\
	\label{eq:new-strong-eq2} s_{pp} \dot{p} + \alpha_p \text{div}(\dot{\boldsymbol{u}}) + s_{p\theta}\dot{\theta} -  \text{div} (\kappa_p \nabla p) &= -m_p && \text{in} \ (0,T] \times \Omega,\\
	\label{eq:new-strong-eq3} s_{\theta\theta} \dot{\theta} + \alpha_\theta \text{div}(\dot{\boldsymbol{u}}) + s_{\theta p}\dot{p} - \text{div}(\kappa_\theta \nabla \theta) &= -m_\theta && \text{in} \ (0,T] \times \Omega,
\end{align}
\end{subequations}
where $\boldsymbol{f}:=\dot{\tilde{\boldsymbol{f}}}$. This reformulation converts the algebraic mechanical equation into a differential equation. 
A solution of \eqref{eq:new-strong-eqs} is also a solution of \eqref{sys-eqs} provided the initial data $(\boldsymbol{u}(0),p(0),\theta(0))$ satisfies the mechanical equilibrium equation \eqref{sys1} at $t=0$, i.e.,
    \[
        -\operatorname{div}(\boldsymbol{\mathcal{C}}\epsilon(\boldsymbol{u}(0)))
        + \alpha_p \nabla p(0) + \alpha_\theta \nabla \theta(0) = \tilde{\boldsymbol{f}}(0).
    \]
In the analysis below we treat \eqref{eq:new-strong-eqs} as the primary state system. The well-posedness, optimal control, and error analysis are stated for this differentiated formulation, and the original system \eqref{sys-eqs} is recovered whenever the above compatibility condition holds.


We assume that the storage--coupling matrix
\begin{align}\label{eq:storage-matrix}
	\boldsymbol{S} := \begin{pmatrix} s_{pp} & s_{p\theta} \\ s_{\theta p} & s_{\theta\theta} \end{pmatrix}
\end{align}
is symmetric positive semidefinite, i.e.\ $s_{pp},s_{\theta\theta}\ge 0$ and $s_{pp}s_{\theta\theta}\ge s_{p\theta}^2$. 
Writing $c:=(\alpha_p,\alpha_\theta)^\top$ and $c^\perp:=(-\alpha_\theta,\alpha_p)^\top$, we assume in addition the \emph{effective storage condition}
\begin{align}\label{eq:S-kernel-condition}
    (c^\perp)^\top \boldsymbol{S}\, c^\perp = \alpha_\theta^2 s_{pp} - 2\alpha_p\alpha_\theta s_{p\theta} + \alpha_p^2 s_{\theta\theta} \ge C_s > 0,
\end{align}
for some $C_s>0$. 
This condition is strictly weaker than $\boldsymbol{S}>0$: it allows $\boldsymbol{S}$ to be degenerate, but not in the direction $c^\perp$ that is invisible to the mechanical coupling $\alpha_p p+\alpha_\theta\theta$. In particular, \eqref{eq:S-kernel-condition} allows the storage-degenerate case $s_{pp}=0$ or $s_{\theta\theta}=0$. The same effective-storage structure will also be used in the discrete analysis, where it is combined with a discrete inf-sup condition for the mechanically visible combination $\alpha_p q_h+\alpha_\theta\psi_h$ (see Section~\ref{Finite element approximation:}).

A broad range of discretization techniques has been developed for
Biot-type poroelasticity, including conforming and nonconforming finite elements, mixed methods, discontinuous Galerkin methods, HDG and embedded-HDG schemes, virtual element methods, hybrid high-order methods, and least-squares formulations
\cite{MR2177147,phillips2008coupling,MR3047799,MR3606362,MR3803860,
MR3907413,MR4659441,MR4221326,MR4636155,MR3504993,MR4405491}.


Numerical methods for thermo-poroelasticity problems have attracted increasing attention in recent years. Brun et al. studied a class of monolithic and splitting schemes in \cite{MR4146798}. Antonietti et al. proposed high-order discontinuous Galerkin approaches on polygonal and polyhedral grids \cite{AntoniettiBonettiBotti2023}. Zhang and Rui \cite{ZhangRui2022} developed a Galerkin method for the fully coupled problem. Yi and Lee \cite{YiLee2024} introduced a physics-preserving enriched Galerkin method for the fully coupled model, proving well-posedness and optimal a priori error estimates together with local conservation properties. Chen et al. analyzed a sequential symmetric interior penalty discontinuous Galerkin scheme for fully coupled quasi-static thermo-poroelasticity problems \cite{ChenCuiZhou2026}.
These works discuss stable discretization methods, iterative solution strategies, and a priori error estimates for forward thermo-poroelastic models. In contrast, optimal control problems governed by fully coupled thermo-poroelasticity appear to have received little attention. The present paper addresses this gap by analyzing a distributed optimal control problem for a three-field thermo-poroelastic system, deriving the associated adjoint optimality system and proving fully discrete error estimates for the state, adjoint, and control variables.

\subsection{Primary contributions}

Optimal control problems arise naturally in many geophysics applications, where controlling subsurface flow processes is important for understanding and influencing systems. One prominent example is the enhanced geothermal systems and the analysis of how multiphysical behaviors of fluid flows and heat diffusion are controlled by fluid injection.  
Despite the relevance of such models, the literature on optimal control problems governed by poroelastic or thermo-poroelastic systems remains rather limited. In this direction, Bociu et al. \cite{MR4410836} studied a linear–quadratic elliptic–parabolic optimal control problem describing fluid flow in deformable porous media, under the assumption of a vanishing specific storage coefficient $s_{0} = 0$. Recently, numerical methods for distributed optimal control problems governed by poroelasticity equations were studied in \cite{KhanLeeSingh2026} via a reformulation to differential equations. 
To the best of our knowledge, the numerical analysis of optimal control problems governed by thermo-poroelastic systems with distributed controls has not yet been investigated in the existing literature. 

The main contribution of this work is threefold. First, we identify an
effective-storage structure for the differentiated thermo-poroelastic
system. The elastic Schur complement controls the mechanically visible
combination of \(\alpha_p p+\alpha_\theta\theta\), 
while the storage matrix
needs to control only the complementary direction \(c^\perp=(-\alpha_\theta,\alpha_p)^T\). 
This yields well-posedness and
stability under the conditions $\boldsymbol S\ge0$ and \eqref{eq:S-kernel-condition} which are weaker than the assumption \(\boldsymbol S>0\). In particular, this establishes a theoretical foundation to analyze thermo-poroelasticity problems with degenerate storage coefficient $s_{pp}=0$ so long as $\boldsymbol S\ge0$ and \eqref{eq:S-kernel-condition} hold.
Second, we carry this effective-storage mechanism to the fully discrete
level. By introducing a discrete effective mass form, we prove stability
and a priori error estimates for appropriate mixed finite elements for spatial discretization and dG(0) scheme for temporal discretization, which is algebraically equivalent to backward Euler on each time interval.
The discrete stability requirement is an inf-sup condition for the
mechanically visible combination
\(\alpha_p q_h+\alpha_\theta\psi_h\), not a separate control of the
pressure and temperature variables through a positive definite storage
matrix.
Third, to the best of our knowledge, this provides one of the first
complete numerical analyses of a distributed optimal control problem
governed by a coupled three-field thermo-poroelastic system with both
fluid-source and heat-source controls. The analysis includes
well-posedness of the state system, derivation of the coupled adjoint
optimality system, fully discrete error estimates of order
\(\mathcal O(h^s+\Delta t)\), and numerical tests illustrating both the
predicted convergence behavior and the storage-degenerate regime covered
by the effective-storage theory.

\subsection{Organization of the paper}
The rest of the paper is organized as follows. Section~\ref{sec:well-posedness} establishes well-posedness of the three-field thermo-poroelastic system, proves existence and uniqueness of the optimal control pair $(\bar{m}_p,\bar{m}_\theta)$, and derives the optimality conditions. Section~\ref{discrete formulation} introduces the fully discrete finite element formulation and the discrete optimality system for both controls. Section~\ref{A priori Error analysis} derives a priori error estimates for the state, adjoint, and controls: the displacement is estimated in a maximum-in-time $H^1(\Omega)$-norm, the pressure and temperature in maximum-in-time $L^2(\Omega)$-norms and space-time $L^2$-norms, the adjoint pressure and temperature in $L^2(0,T;L^2(\Omega))$, and the controls in $L^2(0,T;L^2(\Omega))$, using variational discretization. 
Section~\ref{Numerical Experiments} reports a three-field manufactured-solution validation of the fully discrete optimality system and outlines additional numerical studies.
\section{Well-posedness}\label{sec:well-posedness}
In this section, we discuss well-posedness of the three-field thermo-poroelastic model, existence of an optimal control pair, and the first-order optimality conditions. We begin by introducing the necessary notation.
\subsection{Notations}
For a Banach space $X$ with norm $\|\cdot \|_X$ we will use $L^r(0,t; X)$, $C^0(0,t; X)$ to denote \vspace{-2mm}
\begin{align*}
	\| f \|_{L^r(0,t;X)} &:= \left( \int_0^t \| f(s) \|_X^r \,ds \right)^{1/r} 
	\qquad \text{and} \qquad
	\| f \|_{C^0(0,t;X)} := \sup_{0\le s\le t} \| f(s) \|_X. 
\end{align*}
For functions $v, w \in L^2(D)$ let
\begin{align*}
    (v, w)_{D} := \int_{D} v w \,dx ,
\end{align*}
and this definition is naturally extended to vector-valued or tensor-valued $L^2(D)$ functions $v, w$. 
We consider the following spaces:
\begin{align}
	\boldsymbol{V}&= \{\boldsymbol{v} \in H^1(\Omega; \mathbb{R}^d) \,:\, \boldsymbol{v}|_{\Gamma_d} = 0\},\qquad
	Q= \{ q \in H^1(\Omega) \,:\, q|_{\Gamma_p} = 0 \},\\
	\Psi &= \{ \psi \in H^1(\Omega) \,:\, \psi|_{\Gamma_\theta} = 0 \}.
\end{align}
We use $\boldsymbol{L}^2(\Omega)$ and $\boldsymbol{H}^1(\Omega)$ to emphasize function spaces of vector-valued functions. However, we do not use boldface symbols to denote norms of vector-valued functions, i.e., $\| \cdot \|_{L^2(\Omega)}$ will be used for the $L^2(\Omega)$ norms of scalar- and vector-valued functions.

To formulate the optimal control problem, we introduce two admissible control sets. For the distributed fluid source $m_p$ and the distributed heat source $m_\theta$, we define
\begin{align}
	\label{admcontrol-p}
	\mathcal{M}_{ad}^{p}&:= \{m_p \in L^{2}(0,T;L^{2}(\Omega)): m_{p,a}\le m_p(t,x) \le m_{p,b} \ \text{a.e.\ in} \ (0,T) \times \Omega\},\\
	\label{admcontrol-theta}
	\mathcal{M}_{ad}^{\theta}&:= \{m_\theta \in L^{2}(0,T;L^{2}(\Omega)): m_{\theta,a}\le m_\theta(t,x) \le m_{\theta,b} \ \text{a.e.\ in} \ (0,T) \times \Omega\},
\end{align}
where $m_{p,a} < m_{p,b}$ and $m_{\theta,a} < m_{\theta,b}$ are given real constants, and we set $\mathcal{M}_{ad} := \mathcal{M}_{ad}^p \times \mathcal{M}_{ad}^\theta$.
\subsection{Well-posedness}\label{subsec:well-posedness}

The bilinear forms $a_{\boldsymbol{u}}:\boldsymbol{V} \times \boldsymbol{V} \rightarrow \mathbb{R}$, $b_p:Q \times \boldsymbol{V} \rightarrow \mathbb{R}$, $b_\theta:\Psi\times\boldsymbol{V}\to\mathbb{R}$, $a_p:Q \times Q \to \mathbb{R}$, and $a_\theta:\Psi\times \Psi\to\mathbb{R}$ are defined as
	\begin{align*}
		a_{\boldsymbol{u}}(\boldsymbol{v}, \boldsymbol{w}) &:= (\boldsymbol{\mathcal{C}} \epsilon(\boldsymbol{v}), \epsilon(\boldsymbol{w}))_\Omega, 
        \\
		b_p(q,\boldsymbol{v}) &:= - \alpha_p (q, \dive( \boldsymbol{v}))_\Omega, \qquad b_\theta(\psi,\boldsymbol{v}) := -\alpha_\theta (\psi, \dive(\boldsymbol{v}))_\Omega,
        \\
        a_p(q,r) &:= (\kappa_p \nabla q, \nabla r)_\Omega, \qquad a_\theta(\psi,\zeta) := (\kappa_\theta \nabla\psi,\nabla\zeta)_\Omega.
	\end{align*}
We set 
\begin{align}
    \| \boldsymbol{v}\|_{a_{\boldsymbol{u}}} := (a_{\boldsymbol{u}} (\boldsymbol{v}, \boldsymbol{v}))^{\frac 12}, \qquad     \| q\|_{a_{p}} := (a_{p} (q, q))^{\frac 12},\qquad \|\psi\|_{a_\theta} := (a_\theta(\psi,\psi))^{1/2}.
\end{align}
%

We consider the following variational equations derived from \eqref{eq:new-strong-eqs}: 
\begin{subequations}\label{eq:new-weak-n-eqs}
	\begin{align}
		\label{eq:new-weak-n-eq1} 
        a_{\boldsymbol{u}} (\dot{\boldsymbol{u}}(t), \boldsymbol{v}) + b_p(\dot{p}(t),\boldsymbol{v}) + b_\theta(\dot{\theta}(t),\boldsymbol{v}) &= (\boldsymbol{f}(t), \boldsymbol{v})_{\Omega}, \\
		\label{eq:new-weak-n-eq2} 
		b_p(q,\dot{\boldsymbol{u}}(t)) - (s_{pp} \dot{p}(t), q)_{\Omega} - (s_{p\theta}\dot{\theta}(t),q)_\Omega - a_p(p(t), q) &= (m_p(t), q)_{\Omega}, \\
		\label{eq:new-weak-n-eq3}
		b_\theta(\psi,\dot{\boldsymbol{u}}(t)) - (s_{\theta p}\dot{p}(t),\psi)_\Omega - (s_{\theta\theta} \dot{\theta}(t), \psi)_{\Omega} - a_\theta(\theta(t), \psi) &= (m_\theta(t), \psi)_{\Omega}
	\end{align}
\end{subequations}
for all $(\boldsymbol{v}, q, \psi) \in \boldsymbol{V} \times Q \times \Psi$ and for all $t > 0$; well-posedness of this system is established in Lemma~\ref{lem:effective-storage} below via a mechanical elimination argument, under the effective storage condition \eqref{eq:S-kernel-condition}. 

Throughout the well-posedness and error analysis below, we use the following fact about the divergence operator on $\boldsymbol{V}$, in addition to the effective storage condition \eqref{eq:S-kernel-condition}: since $\boldsymbol{v}\in\boldsymbol{V}$ is constrained only on $\Gamma_d$ with $|\Gamma_d|>0$ and is otherwise free on the traction boundary $\Gamma_t$ with $|\Gamma_t|>0$, the divergence operator $\operatorname{div}:\boldsymbol{V}\to L^2(\Omega)$ is surjective. 
This surjectivity is equivalent to the following inf-sup condition: there exists $C>0$ such that
\begin{equation}\label{eq:continuous-div-infsup}
	\|\chi\|_{L^2(\Omega)} \le C \sup_{0\ne\boldsymbol{v}\in\boldsymbol{V}} \frac{|(\chi,\operatorname{div}\boldsymbol{v})_\Omega|}{\|\boldsymbol{v}\|_{\boldsymbol{V}}} \qquad \forall \chi\in L^2(\Omega).
\end{equation}
%

We prove existence and uniqueness of \eqref{eq:new-weak-n-eqs}.
\begin{lemma}\label{lem:effective-storage}
Assume the effective storage condition \eqref{eq:S-kernel-condition}, and write $X:=Q\times\Psi$, $H:=L^2(\Omega)\times L^2(\Omega)$, so that for $y=(q,\psi)$, $\|y\|_H^2:=\|q\|_{L^2(\Omega)}^2+\|\psi\|_{L^2(\Omega)}^2$ and $\|y\|_X^2:=\|q\|_Q^2+\|\psi\|_\Psi^2$. For $y=(q,\psi)\in H$, define the \emph{elastic lift} $\boldsymbol{u}_y\in\boldsymbol{V}$ by
\begin{equation}\label{eq:elastic-lift}
	a_{\boldsymbol{u}}(\boldsymbol{u}_y,\boldsymbol{v}) + b_p(q,\boldsymbol{v}) + b_\theta(\psi,\boldsymbol{v}) = 0 \qquad \forall \boldsymbol{v}\in\boldsymbol{V};
\end{equation}
by coercivity of $a_{\boldsymbol{u}}$, this has a unique solution satisfying
\begin{equation}\label{eq:elastic-lift-bound}
	\|\boldsymbol{u}_y\|_{\boldsymbol{V}} \le C\bigl(\|q\|_{L^2(\Omega)}+\|\psi\|_{L^2(\Omega)}\bigr),
\end{equation}
so $y\mapsto\boldsymbol{u}_y$ is a bounded linear map $H\to\boldsymbol{V}$. For $\boldsymbol{f}(t)\in\boldsymbol{V}'$, define $\boldsymbol{u}_{\boldsymbol{f}}(t)\in\boldsymbol{V}$ by
\begin{equation}\label{eq:uf-def}
	a_{\boldsymbol{u}}(\boldsymbol{u}_{\boldsymbol{f}}(t),\boldsymbol{v}) = \langle \boldsymbol{f}(t),\boldsymbol{v}\rangle \qquad \forall \boldsymbol{v}\in\boldsymbol{V},
\end{equation}
defined pointwise in $t$ from $\boldsymbol{f}(t)$ alone. Define the effective mass bilinear form on $H$ by
\begin{equation}\label{eq:effective-mass}
   m_{\rm eff}(y,z) := (\boldsymbol{S}y,z)_\Omega + a_{\boldsymbol{u}}(\boldsymbol{u}_y,\boldsymbol{u}_z), \qquad y=(q,\psi),\ z=(q',\psi').
\end{equation}
Then, the bilinear form $m_{\rm eff}$ is symmetric, bounded, and coercive on $H$. Moreover, a triple $(\boldsymbol{u},p,\theta)$ solves \eqref{eq:new-weak-n-eq1}--\eqref{eq:new-weak-n-eq3} for a.e.\ $t\in(0,T)$ if and only if $y:=(p,\theta)$ solves the reduced parabolic problem: find $y=(p,\theta)$ such that, for all $z=(q,\psi)\in Q\times\Psi$,
    \begin{multline}\label{eq:reduced-parabolic}
	   m_{\rm eff}(\dot y,z) + a_p(p,q) + a_\theta(\theta,\psi) 
       \\
       = -(m_p,q)_\Omega - (m_\theta,\psi)_\Omega + b_p(q,\boldsymbol{u}_{\boldsymbol{f}}(t)) + b_\theta(\psi,\boldsymbol{u}_{\boldsymbol{f}}(t)).
    \end{multline}
    and
	\begin{equation}\label{eq:u-reconstruction}
		\dot{\boldsymbol{u}} = \boldsymbol{u}_{\boldsymbol{f}} + \boldsymbol{u}_{\dot y}.
	\end{equation}
\end{lemma}
\begin{proof}
    Boundedness and symmetry of $m_{\rm eff}$ follow from the definition and \eqref{eq:elastic-lift-bound}. Recall the elastic lift $y\mapsto\boldsymbol{u}_y$ defined by \eqref{eq:elastic-lift}. Since $b_p(q,\boldsymbol{v})+b_\theta(\psi,\boldsymbol{v}) = -(\alpha_p q+\alpha_\theta\psi,\operatorname{div}\boldsymbol{v})_\Omega$, the inf-sup condition \eqref{eq:continuous-div-infsup} gives
\begin{equation}\label{eq:elastic-controls-c-direction}
	\|\alpha_p q+\alpha_\theta\psi\|_{L^2(\Omega)} \le C\|\boldsymbol{u}_y\|_{\boldsymbol{V}} .
\end{equation}
By the effective storage condition \eqref{eq:S-kernel-condition}, the quadratic form $\eta\mapsto |c\cdot\eta|^2+\eta^\top\boldsymbol{S}\eta$ on $\mathbb{R}^2$ (with $c=(\alpha_p,\alpha_\theta)^\top$) is positive definite: if it vanishes, then $c\cdot\eta=0$, so $\eta\in\mathrm{span}\{c^\perp\}$, and also $\eta\in\ker\boldsymbol{S}$, which forces $\eta=0$ by \eqref{eq:S-kernel-condition}. Hence there exists $C>0$ with
\[
	|\eta|^2 \le C\bigl(|c\cdot\eta|^2+\eta^\top\boldsymbol{S}\eta\bigr) \qquad \forall\eta\in\mathbb{R}^2,
\]
and integrating over $\Omega$ (with $\eta=(q(x),\psi(x))$) gives
\begin{equation}\label{eq:L2-effective-coercivity}
	\|q\|_{L^2}^2+\|\psi\|_{L^2}^2 \le C\bigl(\|\alpha_p q+\alpha_\theta\psi\|_{L^2}^2 + (\boldsymbol{S}y,y)_\Omega\bigr).
\end{equation}
Combining \eqref{eq:elastic-controls-c-direction} with coercivity of $a_{\boldsymbol{u}}$, $\|\alpha_p q+\alpha_\theta\psi\|_{L^2}^2 \le C\,a_{\boldsymbol{u}}(\boldsymbol{u}_y,\boldsymbol{u}_y)$ 
gives
\begin{equation}\label{eq:meff-coercive}
	m_{\rm eff}(y,y) \ge C_{\rm eff}\bigl(\|q\|_{L^2(\Omega)}^2+\|\psi\|_{L^2(\Omega)}^2\bigr) \qquad \forall y=(q,\psi)\in H.
\end{equation}

For equivalence of \eqref{eq:new-weak-n-eqs} and the reduced problem \eqref{eq:reduced-parabolic}, we now eliminate $\dot{\boldsymbol{u}}$ from \eqref{eq:new-weak-n-eq1} using \eqref{eq:elastic-lift}, recalling $\boldsymbol{u}_{\boldsymbol{f}}(t)$ from \eqref{eq:uf-def}. For $y(t):=(p(t),\theta(t))$, $\dot{\boldsymbol{u}}(t) = \boldsymbol{u}_{\boldsymbol{f}}(t) + \boldsymbol{u}_{\dot y(t)}$, as in \eqref{eq:u-reconstruction}, satisfies \eqref{eq:new-weak-n-eq1} because 
$$a_{\boldsymbol{u}}(\dot{\boldsymbol{u}}(t),\boldsymbol{v}) = a_{\boldsymbol{u}}(\boldsymbol{u}_{\boldsymbol{f}}(t),\boldsymbol{v}) + a_{\boldsymbol{u}}(\boldsymbol{u}_{\dot y(t)},\boldsymbol{v}) = \langle\boldsymbol{f}(t),\boldsymbol{v}\rangle - b_p(\dot p(t),\boldsymbol{v}) - b_\theta(\dot\theta(t),\boldsymbol{v})$$ 
by \eqref{eq:uf-def} and \eqref{eq:elastic-lift}, which rearranges to \eqref{eq:new-weak-n-eq1}. Substituting \eqref{eq:u-reconstruction} into \eqref{eq:new-weak-n-eq2}--\eqref{eq:new-weak-n-eq3} and using the identity $b_p(q,\boldsymbol{u}_{y})+b_\theta(\psi,\boldsymbol{u}_{y}) = -a_{\boldsymbol{u}}(\boldsymbol{u}_y,\boldsymbol{u}_z)$ with $z=(q,\psi)$ from \eqref{eq:elastic-lift}, gives the reduced problem \eqref{eq:reduced-parabolic}.
Conversely, given a solution $y=(p,\theta)$ of \eqref{eq:reduced-parabolic}, defining $\dot{\boldsymbol{u}}$ by \eqref{eq:u-reconstruction} and running the above identities in reverse shows $(\boldsymbol{u},p,\theta)$ solves \eqref{eq:new-weak-n-eq1}--\eqref{eq:new-weak-n-eq3}. 
\end{proof}

\begin{theorem}
    \label{thm:well-posedness}
    Assume the effective storage condition \eqref{eq:S-kernel-condition} holds.
    Given $m_p,m_\theta\in L^2(0,T;L^2(\Omega))$, $\boldsymbol{f}\in L^2(0,T;\boldsymbol{L}^2(\Omega))$, and $(p(0),\theta(0))\in H$, the reduced problem \eqref{eq:reduced-parabolic}
    has a unique solution $(p,\theta)\in L^2(0,T;X)\cap H^1(0,T;X')\cap C([0,T];H)$ satisfying
	\begin{multline}
	    \label{eq:basic-reduced-estimate}
		\|(p,\theta)\|_{L^\infty(0,T;H)} + \|(p,\theta)\|_{L^2(0,T;X)} 
        \\
        \le C\bigl(\|(p(0),\theta(0))\|_{H} + \|m_p\|_{L^2(0,T;L^2)} + \|m_\theta\|_{L^2(0,T;L^2)} + \|\boldsymbol{f}\|_{L^2(0,T;\boldsymbol{L}^2)}\bigr).
	\end{multline}
	If, in addition, $(p(0),\theta(0))\in X$, then $(p,\theta)\in H^1(0,T;H)\cap L^\infty(0,T;X)$ and
	\begin{multline}
        \label{eq:higher-reduced-estimate}
		\|(\dot p,\dot\theta)\|_{L^2(0,T;H)} + \|(p,\theta)\|_{L^\infty(0,T;X)} 
        \\
        \le C\bigl(\|(p(0),\theta(0))\|_{X} + \|m_p\|_{L^2(0,T;L^2)} + \|m_\theta\|_{L^2(0,T;L^2)} + \|\boldsymbol{f}\|_{L^2(0,T;\boldsymbol{L}^2)}\bigr),
	\end{multline}
    Moreover, if $\boldsymbol{u}(0) \in \boldsymbol{V}$ is also given, $\boldsymbol{u}(t) := \boldsymbol{u}(0) + \int_0^t \dot{\boldsymbol{u}}(s)\,ds$ satisfies 
    \begin{multline}        
        \|\boldsymbol{u}\|_{H^1(0,T;\boldsymbol{V})} 
        \\
        \le C\bigl(\|(p(0),\theta(0))\|_{X} + \|\boldsymbol{u}(0)\|_{\boldsymbol{V}} + \|m_p\|_{L^2(0,T;L^2)} + \|m_\theta\|_{L^2(0,T;L^2)} + \|\boldsymbol{f}\|_{L^2(0,T;\boldsymbol{L}^2)}\bigr) .
    \end{multline}
\end{theorem}

\begin{proof}
We prove existence of solutions of \eqref{eq:reduced-parabolic} via the theory of Gelfand triple.
Recall $X=Q\times\Psi$, $H=L^2(\Omega)\times L^2(\Omega)$. By \eqref{eq:meff-coercive}, $m_{\rm eff}$ is an inner product on $H$ equivalent to the standard one. For $y=(p,\theta)$ and $z=(q,\psi)$, set $a_X(y,z):=a_p(p,q)+a_\theta(\theta,\psi)$; this diffusion form is continuous and coercive on $X$. Hence \eqref{eq:reduced-parabolic} is an abstract linear parabolic problem in the Gelfand triple $X\hookrightarrow H\hookrightarrow X'$. Consequently, given $m_p,m_\theta\in L^2(0,T;L^2(\Omega))$, $\boldsymbol{f}\in L^2(0,T;\boldsymbol{L}^2(\Omega))$, 
and $(p(0),\theta(0))\in H$, standard abstract parabolic theory for the Gelfand triple gives a unique solution
\[
	(p,\theta)\in L^2(0,T;X)\cap H^1(0,T;X')\cap C([0,T];H)
\]
satisfying the basic estimate \eqref{eq:basic-reduced-estimate}.

For \eqref{eq:higher-reduced-estimate} note that the right-hand side of \eqref{eq:reduced-parabolic}, as a functional of $z=(q,\psi)$, in fact belongs to $L^2(0,T;H)$. Since $\boldsymbol{u}_{\boldsymbol{f}}(t)\in\boldsymbol{V}\subset\boldsymbol{H}^1(\Omega)$ for a.e.\ $t$, $\operatorname{div}\boldsymbol{u}_{\boldsymbol{f}}\in L^2(0,T;L^2(\Omega))$, so $q\mapsto b_p(q,\boldsymbol{u}_{\boldsymbol{f}}(t))=-\alpha_p(q,\operatorname{div}\boldsymbol{u}_{\boldsymbol{f}}(t))_\Omega$ and $\psi\mapsto b_\theta(\psi,\boldsymbol{u}_{\boldsymbol{f}}(t))$ are represented by the $L^2(\Omega)$-functions $-\alpha_p\operatorname{div}\boldsymbol{u}_{\boldsymbol{f}}(t)$, $-\alpha_\theta\operatorname{div}\boldsymbol{u}_{\boldsymbol{f}}(t)$. Together with $m_p,m_\theta\in L^2(0,T;L^2(\Omega))$, this gives an $H$-valued right-hand side, which is what is needed in the Gelfand triple theory. If, in addition, $(p(0),\theta (0))\in X$, the standard higher energy estimate for symmetric parabolic problems testing \eqref{eq:reduced-parabolic} with $z=\dot y$, using $m_{\rm eff}(\dot y,\dot y)\ge c_{\rm eff}\|\dot y\|_H^2$ and $\tfrac12\tfrac{d}{dt}a_X(y,y)=a_X(y,\dot y)$,  gives $(p,\theta)\in H^1(0,T;H)\cap L^\infty(0,T;X)$ satisfying the higher estimate \eqref{eq:higher-reduced-estimate}. Since $\boldsymbol{f}\in L^2(0,T;\boldsymbol{L}^2(\Omega))$ gives $\|\boldsymbol{u}_{\boldsymbol{f}}\|_{L^2(0,T;\boldsymbol{V})}\le C\|\boldsymbol{f}\|_{L^2(0,T;L^2(\Omega))}$ directly from \eqref{eq:uf-def} and coercivity of $a_{\boldsymbol{u}}$, both \eqref{eq:basic-reduced-estimate} and \eqref{eq:higher-reduced-estimate} have their right-hand sides bounded by data in the norms already used in \eqref{eq:state-rate-estimate}--\eqref{eq:u-value-estimate} below.

Finally, for reconstruction of $\boldsymbol{u}$, define $\dot{\boldsymbol{u}}$ by \eqref{eq:u-reconstruction} and set $\boldsymbol{u}(t):=\boldsymbol{u}(0)+\int_0^t\dot{\boldsymbol{u}}(s)\,ds$ for given $\boldsymbol{u}(0)\in\boldsymbol{V}$. By \eqref{eq:elastic-lift-bound} and \eqref{eq:higher-reduced-estimate},
\begin{align*}
	\|\dot{\boldsymbol{u}}\|_{L^2(0,T;\boldsymbol{V})} &\le \|\boldsymbol{u}_{\boldsymbol{f}}\|_{L^2(0,T;\boldsymbol{V})} + C\|(\dot p,\dot\theta)\|_{L^2(0,T;H)} < \infty,
    \\
    &\le C\bigl(\|(p(0),\theta(0))\|_{X} + \|m_p\|_{L^2(0,T;L^2)} + \|m_\theta\|_{L^2(0,T;L^2)} + \|\boldsymbol{f}\|_{L^2(0,T;\boldsymbol{L}^2)}\bigr),
\end{align*}
so $\boldsymbol{u}\in H^1(0,T;\boldsymbol{V})$. By construction, the triple $(\boldsymbol{u},p,\theta)$ satisfies \eqref{eq:new-weak-n-eq1}--\eqref{eq:new-weak-n-eq3}.
\end{proof}


\begin{lemma}\label{lem:state-energy}
	Suppose that $(\boldsymbol{u}, p, \theta) \in H^1(0,T;\boldsymbol{V}) \times (H^1(0,T;L^2(\Omega))\cap L^2(0,T;Q)) \times (H^1(0,T;L^2(\Omega))\cap L^2(0,T;\Psi))$ satisfy \eqref{eq:new-weak-n-eqs} for a.e.\ $0< t <T$, under the assumptions of Lemma~\ref{lem:effective-storage}. Then, 
\algn{
	\notag
	&\| \dot{\boldsymbol{u}} \|_{L^2(0,t; \boldsymbol{V})} + \|\dot{p}\|_{L^2(0,t; L^2(\Omega))} + \|\dot{\theta}\|_{L^2(0,t;L^2(\Omega))} + \|p\|_{L^\infty(0,t; Q)} + \|\theta\|_{L^\infty(0,t; \Psi)}
	\\
	\label{eq:state-rate-estimate}
	&\quad \le C (\|p(0)\|_{Q} + \|\theta(0)\|_\Psi + \|\boldsymbol{f}\|_{L^2(0,t; L^2(\Omega))} + \|m_p\|_{L^2(0,t;L^2(\Omega))} + \|m_\theta\|_{L^2(0,t;L^2(\Omega))}), 
	\\
    \notag
	&\|{p}\|_{L^\infty(0,t;L^2(\Omega))} + \|\theta\|_{L^\infty(0,t;L^2(\Omega))}
    \\
	\label{eq:state-value-estimate}
	&\quad \le C \bigl(\|p(0)\|_{L^2(\Omega)} + \|\theta(0)\|_{L^2(\Omega)} + \|\boldsymbol{f}\|_{L^2(0,t; L^2(\Omega))} + \|m_p\|_{L^2(0,t;L^2(\Omega))} + \|m_\theta\|_{L^2(0,t;L^2(\Omega))}\bigr),
	\\
    \notag
	&\|{\boldsymbol{u}}\|_{C^0(0,t;\boldsymbol{V})}
    \\
	\label{eq:u-value-estimate}
	&\quad \le \|\boldsymbol{u}(0)\|_{\boldsymbol{V}} + C\sqrt{t}\, (\|p(0)\|_{Q} + \|\theta(0)\|_\Psi + \|\boldsymbol{f}\|_{L^2(0,t; L^2(\Omega))} + \|m_p\|_{L^2(0,t;L^2(\Omega))} + \|m_\theta\|_{L^2(0,t;L^2(\Omega))}) .	
}
\end{lemma}
\begin{proof}
By Theorem~\ref{thm:well-posedness} applied on $(0,t)$, \eqref{eq:higher-reduced-estimate} gives
\begin{multline*}    
	\|\dot p\|_{L^2(0,t;L^2(\Omega))} + \|\dot\theta\|_{L^2(0,t;L^2(\Omega))} + \|p\|_{L^\infty(0,t;Q)} + \|\theta\|_{L^\infty(0,t;\Psi)}
    \\
	\le C\bigl(\|p(0)\|_Q+\|\theta(0)\|_\Psi + \|m_p\|_{L^2(0,t;L^2)} + \|m_\theta\|_{L^2(0,t;L^2)} + \|\boldsymbol{f}\|_{L^2(0,t;L^2)}\bigr),
\end{multline*}
using $\|\boldsymbol{u}_{\boldsymbol{f}}\|_{L^2(0,t;\boldsymbol{V})}\le C\|\boldsymbol{f}\|_{L^2(0,t;L^2(\Omega))}$. Combined with 
\algns{
\|\dot{\boldsymbol{u}}\|_{L^2(0,t;\boldsymbol{V})} \le \|\boldsymbol{u}_{\boldsymbol{f}}\|_{L^2(0,t;\boldsymbol{V})} + C(\|\dot p\|_{L^2(0,t;L^2)}+\|\dot\theta\|_{L^2(0,t;L^2)})
}
from \eqref{eq:u-reconstruction} and \eqref{eq:elastic-lift-bound}, this gives \eqref{eq:state-rate-estimate}.

The estimate for $p,\theta$ in \eqref{eq:state-value-estimate} follows directly from the reduced abstract parabolic equation with the effective mass form $m_{\rm eff}$, 
testing \eqref{eq:reduced-parabolic} with $z=y=(p,\theta)$ itself gives $\tfrac12\tfrac{d}{dt}m_{\rm eff}(y,y) + a_X(y,y) = \ell(t;y)$, and Young's inequality on the right-hand side ($\ell(t;y)\le\tfrac12 a_X(y,y) + C\|\ell(t)\|_{X'}^2$) absorbs the $a_X(y,y)$ term without any time-weighting; integrating in $t$ and using coercivity of $m_{\rm eff}$ (Lemma~\ref{lem:effective-storage}, part (i)) gives exactly \eqref{eq:basic-reduced-estimate} on $(0,t)$, i.e.
\begin{multline*}
	\|p\|_{L^\infty(0,t;L^2)} + \|\theta\|_{L^\infty(0,t;L^2)} 
    \\
    \le C\bigl(\|p(0)\|_{L^2}+\|\theta(0)\|_{L^2} + \|m_p\|_{L^2(0,t;L^2)} + \|m_\theta\|_{L^2(0,t;L^2)} + \|\boldsymbol{u}_{\boldsymbol{f}}\|_{L^2(0,t;\boldsymbol{V})}\bigr),
\end{multline*}
which is \eqref{eq:state-value-estimate} after using $\|\boldsymbol{u}_{\boldsymbol{f}}\|_{L^2(0,t;\boldsymbol{V})}\le C\|\boldsymbol{f}\|_{L^2(0,t;L^2(\Omega))}$ again. The constant is independent of $t$ because the argument bounds $\|y(t)\|_H^2$ directly from $\|y(0)\|_H^2$ and $\|\ell\|_{L^2(0,t;X')}^2$, with no integration of $\dot y$ in time.

For $\boldsymbol{u}$ in \eqref{eq:u-value-estimate}, the fundamental theorem of calculus and the Cauchy--Schwarz inequality give
\begin{align*}
    \|\boldsymbol{u}(t)\|_{\boldsymbol{V}} 
    \le \|\boldsymbol{u}(0)\|_{\boldsymbol{V}}
    + \sqrt{t}\,\|\dot{\boldsymbol{u}}\|_{L^2(0,t;\boldsymbol{V})},
\end{align*}
and \eqref{eq:u-value-estimate} follows by combining with \eqref{eq:state-rate-estimate}. 
\end{proof}

For later use, we state the following elliptic-regularity-type assumption for the solutions of the differentiated thermo-poroelastic system; it is used in Theorem~\ref{thm:well-posedness} to describe additional regularity of the weak solution. The a priori error estimates below impose the specific (typically stronger) regularity assumptions needed in each case independently. There exists $0<s\le 1$ such that, whenever the data are sufficiently regular, the solution satisfies
\begin{align}
    \label{eq:elliptic-regularity}
    &\|\boldsymbol{u}\|_{L^2(0,T;\boldsymbol{H}^{1+s}(\Omega))}
    + \|p\|_{L^2(0,T;H^{1+s}(\Omega))}
    + \|\theta\|_{L^2(0,T;H^{1+s}(\Omega))}
    \\
    \notag
    &\quad \le C_{\rm reg}
    \left(
    \|\boldsymbol{u}(0)\|_{\boldsymbol{H}^{1+s}(\Omega)}
    + \|p(0)\|_{H^1(\Omega)}
    + \|\theta(0)\|_{H^1(\Omega)} \right)
    \\
    \notag
    &\qquad + C_{\rm reg} \left( \|\boldsymbol{f}\|_{L^2(0,T;\boldsymbol{L}^2(\Omega))}
    + \|m_p\|_{L^2(0,T;L^2(\Omega))}
    + \|m_\theta\|_{L^2(0,T;L^2(\Omega))}
    \right).
\end{align}
%
\subsection{Existence of optimal control}\label{Existence of optimal control}
Our main goal is to choose distributed controls $(m_p,m_\theta) \in \mc{M}_{ad}$ so that the corresponding solid displacement $\boldsymbol{u}$, fluid pressure $p$, and temperature $\theta$ are the best possible approximations to desired target states $\boldsymbol{u}_C$, $p_C$, and $\theta_C$, respectively.
\begin{definition} 
	For given $(\boldsymbol{u}_0, p_0, \theta_0) \in \boldsymbol{V} \times Q \times \Psi$, we define that 
	$(\boldsymbol{u}, p, \theta, m_p, m_\theta) \in H^1(0,T;\boldsymbol{V}) \times (H^1(0,T;L^2(\Omega))\cap L^2(0,T;Q)) \times (H^1(0,T;L^2(\Omega))\cap L^2(0,T;\Psi)) \times L^2(0,T;L^2(\Omega))^2$ is admissible for initial data $(\boldsymbol{u}_0, p_0, \theta_0)$ if it satisfies the weak formulation \eqref{eq:new-weak-n-eqs} with controls $(m_p,m_\theta)$.
\end{definition}
Let $\alpha_{\boldsymbol{u}, C}$, $\alpha_{p, C}$, $\alpha_{\theta,C}$ be non-negative numbers such that $\alpha_{\boldsymbol{u}, C} +  \alpha_{p, C} + \alpha_{\theta,C} > 0$, and let $\gamma_p, \gamma_\theta > 0$ be control cost parameters. For given target fields
\begin{align*}
    \boldsymbol{u}_C \in L^2(0,T; \boldsymbol{L}^2(\Omega)), \quad p_C \in L^2(0,T; L^2(\Omega)), \quad \theta_C \in L^2(0,T; L^2(\Omega)),
\end{align*}
and an admissible quintuple $(\boldsymbol{u}, p, \theta, m_p, m_\theta)$, we define the objective functional as
\begin{align}
	\label{obj_fun} J \LRp{\boldsymbol{u}, p, \theta, m_p, m_\theta} &= 
	  \frac{\alpha_{\boldsymbol{u}, C}}{2} \int_0^T \int_{\Omega} |\boldsymbol{u} - \boldsymbol{u}_C|^2 \, dx \, dt 
     + \frac{\alpha_{p, C}}{2} \int_0^T \int_{\Omega} |p - p_C|^2 \, dx \, dt 
     \\ \notag
	 &\quad + \frac{\alpha_{\theta,C}}{2} \int_0^T \int_\Omega |\theta - \theta_C|^2\,dx\,dt
     \\
	 &\quad + \frac{\gamma_p}{2} \int_0^T \int_{\Omega} |m_p|^2 \, dx \, dt + \frac{\gamma_\theta}{2} \int_0^T \int_\Omega |m_\theta|^2\,dx\,dt,
\end{align}
where $\mathcal{M}_{ad} = \mathcal{M}_{ad}^p \times \mathcal{M}_{ad}^\theta$ is the closed convex admissible set defined in \eqref{admcontrol-p}--\eqref{admcontrol-theta}. The regularization terms bound the controls and guarantee existence of an optimal pair. 
\begin{definition}
	$(\bar{m}_p, \bar{m}_\theta) \in \mc{M}_{ad}$ is called an optimal control pair if the admissible quintuple $(\bar{\boldsymbol{u}}, \bar{p}, \bar{\theta}, \bar{m}_p, \bar{m}_\theta)$ for given initial data $(\boldsymbol{u}_0, p_0, \theta_0)$ satisfies
	\algn{
		\nonumber J \LRp{\boldsymbol{u}, p, \theta, m_p, m_\theta} \ge J \LRp{\bar{\boldsymbol{u}}, \bar{p}, \bar{\theta}, \bar{m}_p, \bar{m}_\theta}
	}
	for all $(\boldsymbol{u}, p, \theta, m_p, m_\theta)$ admissible with $(m_p,m_\theta) \in \mc{M}_{ad}$.
\end{definition}
\begin{definition}
	Let $\mathcal{S}:L^2(0,T;L^2(\Omega))^2 \rightarrow L^2(0,T;\boldsymbol{V} \times Q \times \Psi)$ denote the linear control-to-state map that sends $(m_p,m_\theta)$ to $(\boldsymbol{u},p,\theta)$, the weak solution of \eqref{eq:new-strong-eqs} with zero initial data, zero body force, and controls $(m_p,m_\theta)$. We restrict $\mathcal{S}$ to $\mathcal{M}_{ad}$ when considering admissible controls.
\end{definition}
The energy estimate \eqref{eq:state-rate-estimate} shows that $\mathcal{S}$ is linear and bounded. Write the state as $(\boldsymbol{u},p,\theta) = \boldsymbol{y}^0 + \mathcal{S}(m_p,m_\theta)$, where $\boldsymbol{y}^0 = (\boldsymbol{u}^0,p^0,\theta^0)$ solves \eqref{eq:new-strong-eqs} with $(m_p,m_\theta)=(0,0)$ and the given initial data and body force, while $\mathcal{S}(m_p,m_\theta)$ solves the same system with zero initial data, zero body force, and controls $(m_p,m_\theta)$. The reduced cost functional is then
\begin{align}
    \notag
	j(m_p,m_\theta) &= \frac{\gamma_p}{2}\|m_p\|_{L^2(0,T;L^2(\Omega))}^2 + \frac{\gamma_\theta}{2}\|m_\theta\|_{L^2(0,T;L^2(\Omega))}^2
    \\
    \label{eq:reduced-functional}
	&+ \frac{\alpha_{\boldsymbol{u},C}}{2}\|\boldsymbol{u}^0+\mathcal{S}_{\boldsymbol{u}}(m_p,m_\theta)-\boldsymbol{u}_C\|_{L^2(0,T;L^2(\Omega))}^2
    \\
    \notag
    &+ \frac{\alpha_{p,C}}{2}\|p^0+\mathcal{S}_p(m_p,m_\theta)-p_C\|_{L^2(0,T;L^2(\Omega))}^2
    \\
    \notag
    &+ \frac{\alpha_{\theta,C}}{2}\|\theta^0+\mathcal{S}_\theta(m_p,m_\theta)-\theta_C\|_{L^2(0,T;L^2(\Omega))}^2,
\end{align}
where $\mathcal{S}_{\boldsymbol{u}},\mathcal{S}_p,\mathcal{S}_\theta$ denote the state components of $\mathcal{S}$. Since $\mathcal{S}$ is linear and bounded, $j:\mathcal{M}_{ad}\to\mathbb{R}$ is Fréchet differentiable. 

\begin{theorem} \label{existence}
	For any given initial data $(\boldsymbol{u}_0, p_0, \theta_0)$ there exists a unique optimal control pair $(\bar{m}_p, \bar{m}_\theta) \in \mc{M}_{ad}$.
\end{theorem}
\begin{proof}
	We minimize the reduced functional $j(m_p,m_\theta)$ defined in \eqref{eq:reduced-functional} over $\mathcal{M}_{ad}$. Since $j\ge 0$, there exists a minimizing sequence $\{(m_{p,n},m_{\theta,n})\}_{n=1}^\infty \subset \mathcal{M}_{ad}$ with
	\algn{
		\nonumber \lim_{n\to\infty} j(m_{p,n},m_{\theta,n}) = \inf_{\mathcal{M}_{ad}} j =: j_{\inf} \ge 0.
	}
	%
	Since $\{j(m_{p,n},m_{\theta,n})\}$ is bounded, $\{m_{p,n}\}$ and $\{m_{\theta,n}\}$ are bounded in $L^2(0,T;L^2(\Omega))$, with $\|m_{p,n}\|_{L^2(0,T;L^2(\Omega))}^2 \le 2j(m_{p,n},m_{\theta,n})/\gamma_p$ and similarly for $m_{\theta,n}$. Since $\mathcal{M}_{ad}$ is weakly closed and bounded, by the Banach--Alaoglu theorem there exists a subsequence (still denoted $n$) and limits $(\bar{m}_p,\bar{m}_\theta) \in \mathcal{M}_{ad}$ such that
	\begin{align*}
		m_{p,n} \rightharpoonup \bar{m}_p \in L^2(0,T;L^2(\Omega)), \quad m_{\theta,n} \rightharpoonup \bar{m}_\theta \in L^2(0,T;L^2(\Omega)).
	\end{align*}
	Since $\mathcal{S}$ is linear and continuous, $m_{p,n} \rightharpoonup \bar{m}_p$ and $m_{\theta,n} \rightharpoonup \bar{m}_\theta$ imply $\mathcal{S}(m_{p,n},m_{\theta,n}) \rightharpoonup \mathcal{S}(\bar{m}_p,\bar{m}_\theta)$ in the Bochner space, so $(\bar{\boldsymbol{u}},\bar{p},\bar{\theta}) = \boldsymbol{y}^0 + \mathcal{S}(\bar{m}_p,\bar{m}_\theta)$ is the state corresponding to $(\bar{m}_p,\bar{m}_\theta)$. The tracking and regularization norms are weakly lower semi-continuous, hence
	\[
		j(\bar{m}_p,\bar{m}_\theta) \le \liminf_{n\to\infty} j(m_{p,n},m_{\theta,n}) = j_{\inf}.
	\]
	Uniqueness follows because the regularization terms 
    \[
    \frac{\gamma_p}{2}\|m_p\|_{L^2(0,T;L^2(\Omega))}^2 + \frac{\gamma_\theta}{2}\|m_\theta\|_{L^2(0,T;L^2(\Omega))}^2
    \]
    make $j$ strictly convex on $\mathcal{M}_{ad}$.
\end{proof}
\subsection{Necessary optimality conditions}
To derive the first order necessary optimality conditions, we use the adjoint operator of control-to-state operator $\mathcal{S}$. 
The corresponding variational equations for the adjoint problem are given by
\begin{subequations}\label{eq:new-weakadj}
	\begin{align}
		\label{eq:new-weak-adj1} 
		-a_{\boldsymbol{u}} (\dot{\boldsymbol{w}}, \boldsymbol{v}) - b_p(\dot{r},\boldsymbol{v}) - b_\theta(\dot{\phi},\boldsymbol{v}) &=  \alpha_{\boldsymbol{u}, C} ( \boldsymbol{u}-\boldsymbol{u}_C, \boldsymbol{v} )_{\Omega}, \\
		\label{eq:new-weak-adj2} 
		-b_p(q,\dot{\boldsymbol{w}}) + (s_{pp} \dot{r}, q)_{\Omega} + (s_{p\theta}\dot{\phi},q)_\Omega - a_p (r, q) &= \alpha_{p, C}( p-p_C, q)_{\Omega},\\
		\label{eq:new-weak-adj3}
		-b_\theta(\psi,\dot{\boldsymbol{w}}) + (s_{\theta p}\dot{r},\psi)_\Omega + (s_{\theta\theta}\dot{\phi},\psi)_\Omega - a_\theta(\phi,\psi) &= \alpha_{\theta,C}(\theta - \theta_C,\psi)_\Omega,
	\end{align}
\end{subequations}
for all $(\boldsymbol{v},q,\psi) \in \boldsymbol{V} \times Q \times \Psi$. Here $\boldsymbol{w}$, $r$, and $\phi$ are the adjoint displacement, adjoint fluid pressure, and adjoint temperature, respectively. Integration by parts in time leads to the terminal conditions $\boldsymbol{w}(\cdot,T) = \boldsymbol{0}$, $r(\cdot,T) = 0$, $\phi(\cdot,T)=0$. 
\begin{lemma}
	$(\bar{m}_p,\bar{m}_\theta)$ is an optimal control pair if and only if
	\begin{align}
		\label{ctsvi-g}
		\left(\gamma_p \bar{m}_p + r ,m_p-\bar{m}_p\right)_{(0,T) \times \Omega} &\ge 0 \qquad \forall m_p \in \mathcal{M}_{ad}^p,\\
		\label{ctsvi-m}
		\left(\gamma_\theta \bar{m}_\theta + \phi ,m_\theta-\bar{m}_\theta\right)_{(0,T) \times \Omega} &\ge 0 \qquad \forall m_\theta \in \mathcal{M}_{ad}^\theta,
	\end{align}
	where $r$ and $\phi$ are the adjoint fluid pressure and adjoint temperature components of \eqref{eq:new-weakadj}.
\end{lemma}
\begin{proof}
    Let $(\delta m_p, \delta m_\theta) \in L^2(0,T;L^2(\Omega))^2$ be an admissible perturbation direction and let $(\delta\boldsymbol{u},\delta p,\delta\theta) = \mathcal{S}(\delta m_p,\delta m_\theta)$ be the corresponding linearized state. Computing the directional derivative of $j$ via the chain rule and using the adjoint system \eqref{eq:new-weakadj} to transfer the state sensitivity to the control space—by testing \eqref{eq:new-weak-adj1}--\eqref{eq:new-weak-adj3} against $(\delta\boldsymbol{u},\delta p,\delta\theta)$ and integrating by parts in time—one obtains
    \[
        j'(m_p,m_\theta)(\delta m_p, \delta m_\theta)
        = (\gamma_p m_p + r, \delta m_p)_{(0,T)\times\Omega}
        + (\gamma_\theta m_\theta + \phi, \delta m_\theta)_{(0,T)\times\Omega},
    \]
    where the boundary terms in time vanish because the linearized state variables $(\delta\boldsymbol{u},\delta p,\delta\theta)$ have zero initial data and the adjoint variables $(\boldsymbol{w},r,\phi)$ have zero terminal data. Here $(r,\phi)$ are the adjoint pressure and temperature components of the solution to \eqref{eq:new-weakadj}. The adjoint system \eqref{eq:new-weakadj} is a backward-in-time thermo-poroelastic system with terminal data zero and right-hand sides given by the state tracking residuals. Setting $s = T-t$ converts it to a forward-in-time problem of the same type as \eqref{eq:new-strong-eqs}; after this time reversal, the adjoint system \eqref{eq:new-weakadj} has the same structure as the differentiated state system \eqref{eq:new-strong-eqs} with $(\boldsymbol{w},r,\phi)$ playing the role of $(\boldsymbol{u},p,\theta)$ and the tracking residuals $\alpha_{\boldsymbol{u},C}(\boldsymbol{u}-\boldsymbol{u}_C)$, $\alpha_{p,C}(p-p_C)$, $\alpha_{\theta,C}(\theta-\theta_C)$ playing the role of the source terms $\boldsymbol{f}$, $m_p$, $m_\theta$. Existence of the adjoint therefore follows by applying the effective-storage reduction of Lemma~\ref{lem:effective-storage} and the existence argument of Theorem~\ref{thm:well-posedness} to the reversed system, and Lemma~\ref{lem:state-energy} applies to the reversed system to give the energy estimate
    \[
    \begin{aligned}
        &\|\dot{\boldsymbol{w}}\|_{L^2(0,T;\boldsymbol{V})}
        + \|\dot r\|_{L^2(0,T;L^2(\Omega))}
        + \|\dot\phi\|_{L^2(0,T;L^2(\Omega))}
        + \|r\|_{L^\infty(0,T;Q)}
        + \|\phi\|_{L^\infty(0,T;\Psi)}
        \\
        &\le C\bigl(\|\boldsymbol{u}-\boldsymbol{u}_C\|_{L^2(0,T;L^2(\Omega))} + \|p-p_C\|_{L^2(0,T;L^2(\Omega))} + \|\theta-\theta_C\|_{L^2(0,T;L^2(\Omega))}\bigr).
    \end{aligned}
    \] The variational inequalities \eqref{ctsvi-g}--\eqref{ctsvi-m} are then the standard first-order necessary conditions $j'(\bar{m}_p,\bar{m}_\theta)(m_p-\bar{m}_p,m_\theta-\bar{m}_\theta) \ge 0$ for all $(m_p,m_\theta)\in\mathcal{M}_{ad}$.
\end{proof}
	Using the component projections $\mathcal{P}_{\mathcal{M}_{ad}^p}:L^{2}(0,T;L^{2}(\Omega)) \rightarrow \mathcal{M}_{ad}^p$ and $\mathcal{P}_{\mathcal{M}_{ad}^\theta}:L^{2}(0,T;L^{2}(\Omega)) \rightarrow \mathcal{M}_{ad}^\theta$ defined as:
	\begin{align}\label{mmproj}
		\mathcal{P}_{\mathcal{M}_{ad}^p}(\varphi)(t,x)&:= \max\{m_{p,a},\min\{\varphi(t,x), m_{p,b}\}\},\\
		\mathcal{P}_{\mathcal{M}_{ad}^\theta}(\varphi)(t,x)&:= \max\{m_{\theta,a},\min\{\varphi(t,x), m_{\theta,b}\}\},
	\end{align}
the continuous variational inequalities \eqref{ctsvi-g}--\eqref{ctsvi-m} can be expressed as the projection formulas:
		\begin{align}\label{projcts}
			\bar{m}_p = \mathcal{P}_{\mathcal{M}_{ad}^p} \left(-\frac{1}{\gamma_p} r(\bar{m}_p,\bar{m}_\theta)\right), \qquad
			\bar{m}_\theta = \mathcal{P}_{\mathcal{M}_{ad}^\theta} \left(-\frac{1}{\gamma_\theta} \phi(\bar{m}_p,\bar{m}_\theta)\right).
	\end{align}
%

%
\section{Discrete formulation}\label{discrete formulation}
In this section we discuss discretizations of the optimal control problem with finite elements. 

\subsection{Spatial and space-time Discretizations}
To construct the finite element approximation, we consider a family of shape-regular partition $\{\mathcal{T}_h\}$ of $\overline{\Omega}$ into triangles $K$ with diameter $h_K$. 

Consider finite element spaces $\boldsymbol{V}_h \times Q_h \times \Psi_h \subset \boldsymbol{V}\times Q \times \Psi$ where $\boldsymbol{V}_h$, $Q_h$, $\Psi_h$ consist of continuous piecewise polynomial functions conforming to the corresponding boundary conditions on $\Gamma_d$, $\Gamma_p$, $\Gamma_{\theta}$. 
We assume the following inf-sup condition:
there exists $C>0$ independent of $h$ such that for every $\chi_h\in\alpha_p Q_h+\alpha_\theta\Psi_h$,
\algn{
	\label{eq:disc-infsup-combination}
	\|\chi_h\|_{L^2(\Omega)} \le C \sup_{\boldsymbol{v}_h\in\boldsymbol{V}_h} \frac{|(\chi_h,\dive\boldsymbol{v}_h)_\Omega|}{\|\boldsymbol{v}_h\|_{\boldsymbol{V}}}.
}
%
Standard Taylor--Hood-type or MINI-type pairs $(\boldsymbol{V}_h,\tilde{Q}_h)$ satisfy this inf-sup condition. If $\alpha_p$, $\alpha_{\theta}$ are constants on $\Omega$ and $Q_h, \Psi_h \subset \tilde{Q}_h$, then $\alpha_p Q_h + \alpha_{\theta} \Psi_h \subset \tilde{Q}_h$, so \eqref{eq:disc-infsup-combination} is satisfied.

For time discretization let $\Delta t$ represent the time step size, where $T=N\Delta t$ and $N$ is the positive integer. We introduce the notation $t_k=k\Delta t$ and $I_k := (t_{k}, t_{k+1}]$ for $k=0,\cdots,N-1$. For a continuous function $\varphi$ defined on $[0,T]$, let $\varphi^k=\varphi(t_k)$. For a given sequence $\{\varphi^k\}_{k\ge 0}$, the derivative is then approximated as follows:
\[\dt\varphi^{k+1}:=\frac{\varphi^{k+1}-\varphi^k}{\Delta t}.\]
Now, we set
\begin{align*}
	\boldsymbol{V}_{hk} &:= \{\boldsymbol{v} \in L^2(0,T; \boldsymbol{V}_h): \boldsymbol{v}|_{I_{k}} \in \mathbb{P}_{0}(I_k; \boldsymbol{V}_h) \ \text{for} \ k = 0,\cdots,N-1\},
    \\
	Q_{hk} &:= \{q \in L^2(0,T; Q_h) : q|_{I_{k}} \in \mathbb{P}_{0}(I_k; Q_h) \ \text{for} \ k = 0,\cdots,N-1\},
    \\
    \Psi_{hk} &:= \{\psi \in L^2(0,T; \Psi_h) : \psi|_{I_{k}} \in \mathbb{P}_{0}(I_k; \Psi_h) \ \text{for} \ k = 0,\cdots,N-1\},
\end{align*}
which means that elements of $\boldsymbol{V}_{hk}$, $Q_{hk}$, and $\Psi_{hk}$ are piecewise constant in time.
\subsection{Finite element approximation}\label{Finite element approximation:}
In the dG(0) method, on each interval $I_k = (t_k, t_{k+1}]$, the discrete solution is sought as a constant in time, i.e., $(U,P,\Theta)|_{I_k} = (\boldsymbol{u}_h^{k+1}, p_h^{k+1}, \theta_h^{k+1}) \in \boldsymbol{V}_h \times Q_h \times \Psi_h$. The scheme is obtained by integrating the equations over $I_k$ and testing against constant-in-time test functions $(\boldsymbol{v},q,\psi)\in\boldsymbol{V}_h\times Q_h\times\Psi_h$, so that
\[
  \frac{1}{\Delta t}\int_{I_k} \dot{\boldsymbol{u}}\,dt \approx \frac{\boldsymbol{u}_h^{k+1}-\boldsymbol{u}_h^k}{\Delta t} =: \dt \boldsymbol{u}_h^{k+1},
\]
and analogously for $p_h$ and $\theta_h$. Throughout this section the notation
\[
  \LRp{\varphi, v}_{I_k \times \Omega} := \int_{t_k}^{t_{k+1}}\int_{\Omega} \varphi\, v\,dx\,dt = \Delta t \int_\Omega \bar{\varphi}^{k+1} v\,dx,
  \qquad \bar{\varphi}^{k+1} := \frac{1}{\Delta t}\int_{t_k}^{t_{k+1}} \varphi\,dt,
\]
is used consistently for any source function $\varphi$.

\medskip
{\bf Discretization of the forward equation} The numerical solution at the $(k+1)$-th time step is defined inductively as follows: For $0\le k \le N-1$ find $(\boldsymbol{u}_h^{k+1}, p_h^{k+1}, \theta_h^{k+1}) \in \boldsymbol{V}_h \times Q_h \times \Psi_h$ such that
\begin{subequations}\label{eq:fullydiscrete-eqs}
	\algn{
		\label{eq:fullydiscrete-eq1}
		a_{\boldsymbol{u}} \LRp{\dt \boldsymbol{u}_h^{k+1}, \boldsymbol{v}} + b_p \LRp{\dt p_h^{k+1}, \boldsymbol{v}} + b_\theta(\dt \theta_h^{k+1},\boldsymbol{v}) &= \frac{1}{\Delta t} \LRp{{\boldsymbol{f}}, \boldsymbol{v}}_{I_k\times\Omega}, \\
		\label{eq:fullydiscrete-eq2}
		b_p \LRp{q, \dt \boldsymbol{u}_h^{k+1}} - \LRp{s_{pp} \dt p_h^{k+1}, q}_{\Omega} - (s_{p\theta}\dt\theta_h^{k+1},q)_\Omega - a_p\LRp{ p_h^{k+1}, q} &= \frac{1}{\Delta t}\LRp{m_p, q}_{I_k\times\Omega}, \\
		\label{eq:fullydiscrete-eq3}
		b_\theta(\psi,\dt \boldsymbol{u}_h^{k+1}) - (s_{\theta p}\dt p_h^{k+1},\psi)_\Omega - (s_{\theta\theta}\dt\theta_h^{k+1},\psi)_\Omega - a_\theta(\theta_h^{k+1},\psi) &= \frac{1}{\Delta t}\LRp{m_\theta, \psi}_{I_k\times\Omega},
	}    
\end{subequations}
for all $(\boldsymbol{v}, q, \psi) \in \boldsymbol{V}_h \times Q_h \times \Psi_h$. 
\begin{lemma}
    The system \eqref{eq:fullydiscrete-eqs} has a unique solution.     
\end{lemma}
\begin{proof}
To establish the well-posedness of \eqref{eq:fullydiscrete-eqs}, it suffices to show that $\boldsymbol{u}_h^{k+1}=0$, $p_h^{k+1}=0$, $\theta_h^{k+1}=0$ when $\boldsymbol{u}_h^{k}, p_h^{k}, \theta_h^k, \boldsymbol{f}, m_p, m_\theta$ are zero. Setting those data to zero, \eqref{eq:fullydiscrete-eqs} becomes
\begin{subequations}
	\algns{
		a_{\boldsymbol{u}} \LRp{\boldsymbol{u}_h^{k+1}, \boldsymbol{v}} + b_p \LRp{p_h^{k+1}, \boldsymbol{v}} + b_\theta(\theta_h^{k+1},\boldsymbol{v}) = 0,
		\\
		b_p \LRp{q, \boldsymbol{u}_h^{k+1}} - \LRp{s_{pp} p_h^{k+1}, q}_{\Omega} - (s_{p\theta}\theta_h^{k+1},q)_\Omega - \Delta t a_p\LRp{p_h^{k+1}, q} = 0,
		\\
		b_\theta(\psi, \boldsymbol{u}_h^{k+1}) - (s_{\theta p}p_h^{k+1},\psi)_\Omega - (s_{\theta\theta}\theta_h^{k+1},\psi)_\Omega - \Delta t a_\theta(\theta_h^{k+1},\psi) = 0,
	}    
\end{subequations}
for all $(\boldsymbol{v}, q, \psi) \in \boldsymbol{V}_h \times Q_h \times \Psi_h$. Choosing $\boldsymbol{v}=\boldsymbol{u}_h^{k+1}$, $q=-p_h^{k+1}$, $\psi=-\theta_h^{k+1}$ and adding all three equations gives
\begin{align*}
	a_{\boldsymbol{u}}(\boldsymbol{u}_h^{k+1},\boldsymbol{u}_h^{k+1}) + (\boldsymbol{S}y_h^{k+1},y_h^{k+1})_\Omega + \Delta t\, a_p(p_h^{k+1},p_h^{k+1}) + \Delta t\, a_\theta(\theta_h^{k+1},\theta_h^{k+1}) = 0,
\end{align*}
where $y_h^{k+1}:=(p_h^{k+1},\theta_h^{k+1})$ and 
\[
(\boldsymbol{S}y_h^{k+1},y_h^{k+1})_\Omega = (s_{pp} p_h^{k+1}, p_h^{k+1})_\Omega + 2(s_{p\theta}\theta_h^{k+1},p_h^{k+1})_\Omega + (s_{\theta\theta}\theta_h^{k+1},\theta_h^{k+1})_\Omega .
\]
All four terms on the left are non-negative because $a_{\boldsymbol{u}}$ is coercive, $\boldsymbol{S}\ge0$, and $a_p,a_\theta$ are coercive. This forces $\boldsymbol{u}_h^{k+1}=p_h^{k+1}=\theta_h^{k+1}=0$. 
\end{proof}


{\bf Discrete optimal control problem} Consider the discrete optimal control problem: Find $(m_p,m_\theta) \in \mathcal{M}_{ad}$ such that 
	\begin{align}
	\label{discr_obj} 
    &\min_{(m_p,m_\theta) \in \mathcal{M}_{ad}} J_h \LRp{\boldsymbol{u}_h, p_h, \theta_h, m_p, m_\theta} 
    \\
    \notag
    &:= \frac{\alpha_{\boldsymbol{u}, C} }{2} \sum_{k=0}^{N-1} \|\boldsymbol{u}_h^{k+1} - \boldsymbol{u}_C\|_{L^{2}(I_k; L^2(\Omega))}^2 +\frac{\alpha_{p, C} }{2}\sum_{k=0}^{N-1} \|p_h^{k+1} - p_C \|_{L^{2}(I_k; L^2(\Omega))}^2
	  \\ \notag
	  &+ \frac{\alpha_{\theta, C} }{2}\sum_{k=0}^{N-1} \|\theta_h^{k+1} - \theta_C \|_{L^{2}(I_k; L^2(\Omega))}^2
          +\frac{\gamma_p}{2} \|m_p\|_{L^{2}(0,T; L^2(\Omega))}^2 + \frac{\gamma_\theta}{2}\|m_\theta\|_{L^{2}(0,T;L^2(\Omega))}^2,
	\end{align}
    with $(\boldsymbol{u}_h, p_h, \theta_h) \in \boldsymbol{V}_{hk} \times Q_{hk} \times \Psi_{hk}$ 
	subject to the fully discrete state scheme \eqref{eq:fullydiscrete-eqs}.
    We use the variational discretization approach of \cite{MR2122182} for both controls $(m_p,m_\theta)$, i.e., we do not fix a finite dimensional approximation of the control spaces. For $(m_p,m_\theta)\in\mathcal{M}_{ad}$, let $(\boldsymbol{u}_h(m_p,m_\theta),p_h(m_p,m_\theta),\theta_h(m_p,m_\theta))$ denote the (unique) solution of the discrete state scheme \eqref{eq:fullydiscrete-eqs} driven by these controls, and define the \emph{reduced discrete functional}
    \[
    	j_h(m_p,m_\theta) := J_h\bigl(\boldsymbol{u}_h(m_p,m_\theta), p_h(m_p,m_\theta), \theta_h(m_p,m_\theta), m_p, m_\theta\bigr),
    \]
    i.e., \eqref{discr_obj} with the discrete state eliminated in favor of the controls, so that the discrete OCP \eqref{eq:fullydiscrete-eqs}--\eqref{discr_obj} is equivalent to minimizing $j_h$ over $\mathcal{M}_{ad}$ (the explicit expanded formula for $j_h$ is given in Lemma~\ref{Lemma:5.7} below).
\begin{theorem}
	There exists a unique solution in $\boldsymbol{V}_{hk} \times Q_{hk} \times \Psi_{hk} \times \mathcal{M}_{ad}$ to the discrete optimal control problem \eqref{eq:fullydiscrete-eqs}-\eqref{discr_obj}.
\end{theorem}
\begin{proof}
	The discrete state associated with each admissible control pair $(m_p,m_\theta)\in\mathcal{M}_{ad}$ is uniquely defined by the discrete well-posedness result established above. Hence the discrete OCP can be written as the minimization of the reduced discrete functional $j_h(m_p,m_\theta)$, defined above, over $\mathcal{M}_{ad}$. The terms $\gamma_p\|m_p\|^2/2+\gamma_\theta\|m_\theta\|^2/2$ make every minimizing sequence bounded in $L^2(0,T;L^2(\Omega))^2$. Since $\mathcal{M}_{ad}$ is closed, convex, and bounded, hence weakly closed, a subsequence converges weakly to an admissible pair $(\bar{m}_p,\bar{m}_\theta)\in\mathcal{M}_{ad}$. The discrete control-to-state map is linear and bounded, 
    so the corresponding discrete states converge weakly in the finite-dimensional state space $\boldsymbol{V}_{hk}\times Q_{hk}\times\Psi_{hk}$, which for fixed $h,k$ is a finite-dimensional space. Weak lower semicontinuity of the tracking and regularization terms then gives existence of a minimizer. Strict convexity of $j_h$, due to $\gamma_p,\gamma_\theta>0$, gives uniqueness.
\end{proof}
For the backward dG(0) adjoint, we use the convention
\[
    (\boldsymbol{w}_h,r_h,\phi_h)|_{I_k} = (\boldsymbol{w}_h^k,r_h^k,\phi_h^k),
\]
whereas the forward state is represented on $I_k$ by
\[
    (\boldsymbol{u}_h,p_h,\theta_h)|_{I_k} = (\boldsymbol{u}_h^{k+1},p_h^{k+1},\theta_h^{k+1}).
\]
This convention is the natural one obtained by transposing the forward dG(0) scheme: the adjoint equation on $I_k$ is driven by the state residual on the same interval, $(\boldsymbol{u}_h^{k+1}-\boldsymbol{u}_C,p_h^{k+1}-p_C,\theta_h^{k+1}-\theta_C)$, while the adjoint unknowns themselves are indexed one step earlier, at $k$, consistently with marching backward from the terminal condition at $k=N$. The discrete variational inequalities and the projection formula below accordingly use $r_h^k$, $\phi_h^k$ on $I_k$.
\begin{theorem}
    A control pair $(m_p,m_\theta) \in \mc{M}_{ad}$ and the corresponding state $(\boldsymbol{u}_h, p_h, \theta_h) \in \boldsymbol{V}_{hk} \times Q_{hk} \times \Psi_{hk}$ solve the discrete OCP \eqref{discr_obj}, \eqref{eq:fullydiscrete-eqs} if and only if there exists a discrete adjoint state $(\boldsymbol{w}_h, r_h, \phi_h) \in \boldsymbol{V}_{hk} \times Q_{hk} \times \Psi_{hk}$ satisfying:
	\begin{subequations}
		\label{eq:fullydiscreteadj-eqs11}
		\begin{align}
			\label{eq:fullydiscreteadj-eq12}
			-a_{\boldsymbol{u}} \LRp{\dt \boldsymbol{w}_h^{k+1}, \boldsymbol{v}} - b_p \LRp{\dt r_h^{k+1}, \boldsymbol{v}} - b_\theta(\dt \phi_h^{k+1},\boldsymbol{v}) &= \frac{\alpha_{\boldsymbol{u}, C}}{\Delta t} (\boldsymbol{u}_h^{k+1} - \boldsymbol{u}_{C}, \boldsymbol{v})_{I_k \times \Omega},  
			\\
			\label{eq:fullydiscreteadj-eq23}
			-b_p \LRp{q, \dt \boldsymbol{w}_h^{k+1}} + \LRp{s_{pp} \dt r_h^{k+1}, q}_{\Omega} + (s_{p\theta}\dt \phi_h^{k+1},q)_\Omega - a_p\LRp{ r_h^{k}, q} &= \frac{\alpha_{p, C}}{\Delta t} (p_h^{k+1} - p_C, q)_{I_k \times \Omega},\\
			\label{eq:fullydiscreteadj-eq34}
			-b_\theta(\psi,\dt \boldsymbol{w}_h^{k+1}) + (s_{\theta p}\dt r_h^{k+1},\psi)_\Omega + (s_{\theta\theta}\dt \phi_h^{k+1},\psi)_\Omega - a_\theta(\phi_h^k,\psi) &= \frac{\alpha_{\theta,C}}{\Delta t}(\theta_h^{k+1}-\theta_C,\psi)_{I_k\times\Omega},
		\end{align}
	\end{subequations}
	for all $(\boldsymbol{v}, q, \psi) \in \boldsymbol{V}_h \times Q_h \times \Psi_h$, $k=0,1,\ldots,N-1$, with terminal conditions $\boldsymbol{w}_h^N = 0$, $r_h^N = 0$, $\phi_h^N=0$, and the discrete variational inequalities:
		\begin{align}
			\label{eq:varineq-g}
			\sum_{k=0}^{N-1}(\gamma_p m_p + r_h^{k}, \tilde{m}_p - m_p)_{I_k\times \Omega} &\ge 0 \qquad \forall \tilde{m}_p \in \mc{M}_{ad}^p, \\
			\label{eq:varineq-m}
			\sum_{k=0}^{N-1}(\gamma_\theta m_\theta + \phi_h^{k},\tilde{m}_\theta - m_\theta)_{I_k\times \Omega} &\ge 0 \qquad \forall \tilde{m}_\theta \in \mc{M}_{ad}^\theta.
		\end{align}
\end{theorem}
\begin{proof}
	The result follows by differentiating the finite-dimensional Lagrangian associated with \eqref{discr_obj} and \eqref{eq:fullydiscrete-eqs}. Collecting the coefficients of arbitrary state variations $(\delta\boldsymbol{u}_h^{k+1},\delta p_h^{k+1},\delta\theta_h^{k+1})$, $k=0,\ldots,N-1$, yields the discrete adjoint system \eqref{eq:fullydiscreteadj-eqs11}, while collecting the coefficients of arbitrary control variations $(\delta m_p,\delta m_\theta)\in\mathcal{M}_{ad}-(m_p,m_\theta)$ yields the discrete variational inequalities \eqref{eq:varineq-g}--\eqref{eq:varineq-m}. The index convention follows from transposing the dG(0) forward scheme: the state is represented on $I_k$ by the right endpoint value $(\boldsymbol{u}_h^{k+1},p_h^{k+1},\theta_h^{k+1})$, whereas the adjoint variable paired with the state residual on $I_k$ is $(\boldsymbol{w}_h^k,r_h^k,\phi_h^k)$; consequently the diffusion terms $a_p(r_h^k,\cdot)$, $a_\theta(\phi_h^k,\cdot)$ in \eqref{eq:fullydiscreteadj-eqs11} appear at level $k$, one step earlier than the differenced terms $\dt \boldsymbol{w}_h^{k+1}$, $\dt r_h^{k+1}$, $\dt \phi_h^{k+1}$, and this same convention is what makes $r_h^k$, $\phi_h^k$ (rather than $r_h^{k+1}$, $\phi_h^{k+1}$) the correct terms to appear in the projection formula and variational inequality. 
\end{proof}
The adjoint state system \eqref{eq:fullydiscreteadj-eqs11} can be written in the explicit stepping form:
\begin{subequations}
	\label{eq:fullydiscreteadj-eqs11a}
	\begin{align}
		\label{eq:fullydiscreteadj-eq12a}
		  &a_{\boldsymbol{u}} \LRp{\boldsymbol{w}_h^{k}, \boldsymbol{v}} + b_p \LRp{r_h^{k}, \boldsymbol{v}} + b_\theta(\phi_h^k,\boldsymbol{v}) 
          \\
          \notag
          &=  \alpha_{\boldsymbol{u}, C}  (\boldsymbol{u}_h^{k+1} - \boldsymbol{u}_{C}, \boldsymbol{v})_{I_k \times \Omega} + a_{\boldsymbol{u}} \LRp{\boldsymbol{w}_h^{k+1}, \boldsymbol{v}} + b_p \LRp{r_h^{k+1}, \boldsymbol{v}} + b_\theta(\phi_h^{k+1},\boldsymbol{v}),
		\\
		\label{eq:fullydiscreteadj-eq23a}
		&b_p \LRp{q,\boldsymbol{w}_h^{k}} - \LRp{s_{pp} r_h^{k}, q}_{\Omega} - (s_{p\theta}\phi_h^k,q)_\Omega - \Delta t a_p\LRp{ r_h^{k}, q} 
        \\
        \notag
        &= \alpha_{p, C}  (p_h^{k+1} - p_C, q)_{I_k \times \Omega}  + b_p \LRp{q, \boldsymbol{w}_h^{k+1}} - \LRp{s_{pp} r_h^{k+1}, q}_{\Omega} - (s_{p\theta}\phi_h^{k+1},q)_\Omega,\\
		\label{eq:fullydiscreteadj-eq34a}
		&b_\theta(\psi,\boldsymbol{w}_h^k) - (s_{\theta p}r_h^k,\psi)_\Omega - (s_{\theta\theta} \phi_h^k,\psi)_\Omega - \Delta t a_\theta(\phi_h^k,\psi) 
        \\
        \notag 
        &= \alpha_{\theta,C}(\theta_h^{k+1}-\theta_C,\psi)_{I_k\times\Omega} 
		+ b_\theta(\psi,\boldsymbol{w}_h^{k+1}) - (s_{\theta p}r_h^{k+1},\psi)_\Omega - (s_{\theta\theta} \phi_h^{k+1},\psi)_\Omega,
	\end{align}    
\end{subequations}
for $k=0,1,\ldots,N-1$.
 The variational inequalities \eqref{eq:varineq-g}--\eqref{eq:varineq-m} combined with the state and adjoint systems \eqref{eq:fullydiscrete-eqs} and \eqref{eq:fullydiscreteadj-eqs11a}, respectively, form a fully discrete optimality system for the optimal control problem. The discrete control variable is obtained by employing a discrete version of the projection formula \eqref{mmproj} as (see \cite[Section~4]{MR2644299} for more details):
 \begin{align}\label{proj}
	m_p|_{I_k\times\Omega} = \mathcal{P}_{\mathcal{M}_{ad}^p}\!\left(-\frac{1}{\gamma_p}r_h^{k}\right), \quad
	m_\theta|_{I_k\times\Omega} = \mathcal{P}_{\mathcal{M}_{ad}^\theta}\!\left(-\frac{1}{\gamma_\theta}\phi_h^{k}\right) \quad \text{for} \ k=0,1,\ldots,N-1.
 \end{align} 
The solutions $m_p$ and $m_\theta$ are piecewise constant on each $I_k$, so $m_p^k$ and $m_\theta^k$ are well-defined.
 

\section{A priori error analysis}\label{A priori Error analysis}
In this section, we derive the a priori error estimates of the fully discrete scheme proposed in Section~\ref{discrete formulation}. Throughout this section, the right-hand side $\boldsymbol{f}$, the initial data $(\boldsymbol{u}(0), p(0), \theta(0)) \in \boldsymbol{V} \times Q \times \Psi$, the observation data $\boldsymbol{u}_C$, $p_C$, $\theta_C$, and the coefficients $\alpha_{\boldsymbol{u},C}$, $\alpha_{p,C}$, $\alpha_{\theta,C}$, $\gamma_p$, $\gamma_\theta$ are fixed. All discrete forward equations use the same numerical initial data $(\boldsymbol{u}_h^0, p_h^0, \theta_h^0) \in \boldsymbol{V}_h \times Q_h \times \Psi_h$ satisfying 
\begin{multline}
    \label{eq:initial-data-assumption}
    \|\boldsymbol{u}(0)-\boldsymbol{u}_h^0\|_{H^1(\Omega)}+\|p(0)-p_h^0\|_{H^1(\Omega)}+\|\theta(0)-\theta_h^0\|_{H^1(\Omega)}
    \\
    \le C h^{s}\bigl(\|\boldsymbol{u}(0)\|_{H^{1+s}(\Omega)}+\|p(0)\|_{H^{1+s}(\Omega)}+\|\theta(0)\|_{H^{1+s}(\Omega)}\bigr), \quad 0< s \le 1.
\end{multline}
Note that $(\boldsymbol{u}_h^0, p_h^0, \theta_h^0)$ do not need to satisfy the mechanical equilibrium \eqref{sys1} at $t=0$. 

\subsection{Analysis of auxiliary forward equation}
Consistent with the well-posedness argument of Section~\ref{sec:well-posedness}, the discrete forward error estimate below is derived using a \emph{discrete effective mass} form, 
given the discrete inf-sup condition \eqref{eq:disc-infsup-combination}.

For $y_h=(p_h,\theta_h)\in Q_h\times\Psi_h$, define the discrete elastic lift $\boldsymbol{u}_{y,h}\in\boldsymbol{V}_h$ by
\begin{equation}\label{eq:disc-elastic-lift}
	a_{\boldsymbol{u}}(\boldsymbol{u}_{y,h},\boldsymbol{v}_h) + b_p(p_h,\boldsymbol{v}_h) + b_\theta(\theta_h,\boldsymbol{v}_h) = 0 \qquad \forall\boldsymbol{v}_h\in\boldsymbol{V}_h,
\end{equation}
and the discrete effective mass form
\begin{equation}\label{eq:disc-eff-mass}
	m_{{\rm eff},h}(y_h,z_h) := (\boldsymbol{S}y_h,z_h)_\Omega + a_{\boldsymbol{u}}(\boldsymbol{u}_{y,h},\boldsymbol{u}_{z,h}), \qquad y_h,z_h\in Q_h\times\Psi_h.
\end{equation}

\begin{lemma}\label{lem:disc-effective-coercivity}
	There exists $C_{\rm eff}>0$, independent of $h$, such that
	\begin{equation}\label{eq:disc-meff-coercive}
		m_{{\rm eff},h}(y_h,y_h) \ge C_{\rm eff}\bigl(\|p_h\|_{L^2(\Omega)}^2+\|\theta_h\|_{L^2(\Omega)}^2\bigr) \qquad \forall y_h=(p_h,\theta_h)\in Q_h\times\Psi_h.
	\end{equation}
\end{lemma}
\begin{proof}
By the effective storage condition \eqref{eq:S-kernel-condition}, there exists $C>0$ such that $|\eta|^2\le C(|c\cdot\eta|^2+\eta^\top\boldsymbol{S}\eta)$ for all $\eta\in\mathbb{R}^2$. Applying this pointwise to $\eta=(p_h(x),\theta_h(x))$ and integrating,
\begin{equation}\label{eq:pointwise-to-L2-eff}
	\|p_h\|_{L^2}^2+\|\theta_h\|_{L^2}^2 \le C\bigl(\|\alpha_p p_h+\alpha_\theta\theta_h\|_{L^2}^2 + (\boldsymbol{S}y_h,y_h)_\Omega\bigr).
\end{equation}
By \eqref{eq:disc-infsup-combination} and \eqref{eq:disc-elastic-lift},
\begin{align*}
	\|\alpha_p p_h+\alpha_\theta\theta_h\|_{L^2} &\le C\sup_{\boldsymbol{v}_h\in\boldsymbol{V}_h}\frac{|b_p(p_h,\boldsymbol{v}_h)+b_\theta(\theta_h,\boldsymbol{v}_h)|}{\|\boldsymbol{v}_h\|_{\boldsymbol{V}}} 
    \\
    &= C\sup_{\boldsymbol{v}_h\in\boldsymbol{V}_h}\frac{|a_{\boldsymbol{u}}(\boldsymbol{u}_{y,h},\boldsymbol{v}_h)|}{\|\boldsymbol{v}_h\|_{\boldsymbol{V}}} \le C\|\boldsymbol{u}_{y,h}\|_{\boldsymbol{V}},
\end{align*}
so, using coercivity of $a_{\boldsymbol{u}}$, $\|\alpha_p p_h+\alpha_\theta\theta_h\|_{L^2}^2 \le C\,a_{\boldsymbol{u}}(\boldsymbol{u}_{y,h},\boldsymbol{u}_{y,h})$. Combining with \eqref{eq:pointwise-to-L2-eff} gives \eqref{eq:disc-meff-coercive}.
\end{proof}
Let us denote $\Pi_h^{\boldsymbol{u}} : \boldsymbol{V} \to \boldsymbol{V}_h$, $\Pi_h^{p}: Q \to Q_h$, and $\Pi_h^{\theta}: \Psi \to \Psi_h$ the elliptic projections defined by 
\begin{align}
    \label{eq:elasticity-elliptic-projection}
    a_{\boldsymbol{u}} (\Pi_h^{\boldsymbol{u}} \boldsymbol{v}, \boldsymbol{w}) &=     a_{\boldsymbol{u}} ( \boldsymbol{v}, \boldsymbol{w}) \qquad \forall \boldsymbol{w} \in \boldsymbol{V}_h,
    \\
    \label{eq:poisson-elliptic-projection}
    a_p( \Pi_h^p q, r) &= a_p( q, r) \qquad \forall r \in Q_h, \\
    \label{eq:theta-elliptic-projection}
    a_\theta(\Pi_h^\theta \psi, \zeta) &= a_\theta(\psi, \zeta) \qquad \forall \zeta \in \Psi_h.
\end{align}
It is well-known that 
\begin{align}
    \label{eq:elliptic-projection-approx1}   
    &\| \boldsymbol{v} - \Pi_h^{\boldsymbol{u}} \boldsymbol{v} \|_{a_{\boldsymbol{u}}} \le Ch^{s} \|\boldsymbol{v}\|_{H^{1+s}(\Omega)}, 
    \qquad 
    \| q - \Pi_h^{p} q \|_{a_{p}} \le Ch^{s} \|q\|_{H^{1+s}(\Omega)},
    \\
    \label{eq:elliptic-projection-approx2}   
    &\|\psi - \Pi_h^\theta \psi\|_{a_\theta} \le Ch^s \|\psi\|_{H^{1+s}(\Omega)}
\end{align}
for $0\le s\le 1$, up to the approximation order of the finite element spaces.
\begin{lemma}\label{lemma:aux-forward-error}
	Let $(\boldsymbol{u}, p, \theta, \bar{m}_p, \bar{m}_\theta)$ be the solution of the optimal control problem \eqref{obj_fun}. Suppose that $(\boldsymbol{u}_h(\bar{m}_p,\bar{m}_\theta), p_h(\bar{m}_p,\bar{m}_\theta), \theta_h(\bar{m}_p,\bar{m}_\theta)) \in \boldsymbol{V}_{hk} \times Q_{hk} \times \Psi_{hk}$ solves \eqref{eq:fullydiscrete-eqs} for the continuous optimal controls $(\bar{m}_p,\bar{m}_\theta)$ with numerical initial data $(\boldsymbol{u}_h^0, p_h^0, \theta_h^0)$ satisfying \eqref{eq:initial-data-assumption} for some $0<s\le 1$. Assume also that 
    \[
        \boldsymbol{u}\in H^1(0,T;\boldsymbol{H}^{1+s}(\Omega)), \quad  p\in H^1(0,T;H^{1+s}(\Omega)), \quad \theta\in H^1(0,T;H^{1+s}(\Omega)).
    \]
    Then, 
	\begin{align}
        \label{eq:forward-aux-estimate}
		&\max_{1\le k \le N} \LRp{\| \boldsymbol{u}^k - \boldsymbol{u}_h^k\|_{H^1(\Omega)} + \| p^k - p_h^k\|_{L^2(\Omega)} + \|\theta^k - \theta_h^k\|_{L^2(\Omega)}} 
        \\
        \notag
		&+ \| \boldsymbol{u} - \boldsymbol{u}_h\|_{L^2(0,T; H^1(\Omega))} + \| p - p_h\|_{L^2(0,T; L^2(\Omega))} + \|\theta - \theta_h\|_{L^2(0,T;L^2(\Omega))}
        \\
        \notag
        &\le h^{s}(\|\boldsymbol{u}(0)\|_{H^{1+s}(\Omega)}+\|p(0)\|_{H^{1+s}(\Omega)}+\|\theta(0)\|_{H^{1+s}(\Omega)}) 
        \\
        \notag
        &\quad + C(h^s+\Delta t)
        \left(
        \|\boldsymbol{u}\|_{H^1(0,T;H^{1+s}(\Omega))}
        +
        \|p\|_{H^1(0,T;H^{1+s}(\Omega))}
        +
        \|\theta\|_{H^1(0,T;H^{1+s}(\Omega))}
        \right) .
    \end{align}
\end{lemma}
\begin{proof}
In this proof we write $\dt\phi^k := (\phi^k-\phi^{k-1})/\Delta t$ for $k=1,\ldots,N$, corresponding to the interval $(t_{k-1},t_k]$. We decompose each error into an interpolation part and a discrete correction using the elliptic projections:
\begin{align*}
    e_u^j &:= \boldsymbol{u}^j - \boldsymbol{u}_h^j = e_u^{I,j} + e_u^{h,j}, &
    e_u^{I,j} &:= \boldsymbol{u}^j - \Pi_h^{\boldsymbol{u}}\boldsymbol{u}^j, &
    e_u^{h,j} &:= \Pi_h^{\boldsymbol{u}}\boldsymbol{u}^j - \boldsymbol{u}_h^j, \\
    e_p^j &:= p^j - p_h^j = e_p^{I,j} + e_p^{h,j}, &
    e_p^{I,j} &:= p^j - \Pi_h^p p^j, &
    e_p^{h,j} &:= \Pi_h^p p^j - p_h^j, \\
    e_\theta^j &:= \theta^j - \theta_h^j = e_\theta^{I,j} + e_\theta^{h,j}, &
    e_\theta^{I,j} &:= \theta^j - \Pi_h^\theta \theta^j, &
    e_\theta^{h,j} &:= \Pi_h^\theta \theta^j - \theta_h^j,
\end{align*}
for $j = 0,1,\ldots,N$.
Using the continuous equations \eqref{eq:new-weak-n-eqs} integrated over $I_{k-1}=(t_{k-1},t_k]$ and the definitions of the elliptic projections \eqref{eq:elasticity-elliptic-projection}--\eqref{eq:theta-elliptic-projection}, one derives the following three error equations for $(e_u^{h,k},e_p^{h,k},e_\theta^{h,k})$: for all $(\boldsymbol{v},q,\psi)\in\boldsymbol{V}_h\times Q_h\times\Psi_h$,
\begin{align}
\label{eq:err-1}
a_{\boldsymbol{u}}(\dt e_u^{h,k},\boldsymbol{v})
+ b_p(\dt e_p^{h,k},\boldsymbol{v})
+ b_\theta(\dt e_\theta^{h,k},\boldsymbol{v})
&= F_u^k(\boldsymbol{v}),
\\
\label{eq:err-2}
b_p(q,\dt e_u^{h,k})
- (s_{pp}\dt e_p^{h,k},q)_\Omega
- (s_{p\theta}\dt e_\theta^{h,k},q)_\Omega
- a_p(e_p^{h,k},q)
&= F_p^k(q),
\\
\label{eq:err-3}
b_\theta(\psi,\dt e_u^{h,k})
- (s_{\theta p}\dt e_p^{h,k},\psi)_\Omega
- (s_{\theta\theta}\dt e_\theta^{h,k},\psi)_\Omega
- a_\theta(e_\theta^{h,k},\psi)
&= F_\theta^k(\psi),
\end{align}
where the residuals are
\begin{align*}
    F_u^k(\boldsymbol{v}) &:= -b_p(\dt e_p^{I,k},\boldsymbol{v}) - b_\theta(\dt e_\theta^{I,k},\boldsymbol{v}), \\
    F_p^k(q)      &:= -b_p(q,\dt e_u^{I,k}) + (s_{pp}\dt e_p^{I,k},q)_\Omega + (s_{p\theta}\dt e_\theta^{I,k},q)_\Omega + D_p^k(q), \\
    F_\theta^k(\psi) &:= -b_\theta(\psi,\dt e_u^{I,k}) + (s_{\theta p}\dt e_p^{I,k},\psi)_\Omega + (s_{\theta\theta}\dt e_\theta^{I,k},\psi)_\Omega + D_\theta^k(\psi),
\end{align*}
and the diffusion time residuals are
\begin{align*}
    D_p^k(q) &:= \frac{1}{\Delta t}\int_{t_{k-1}}^{t_k} a_p(p(s)-p^k,q)\,ds, \qquad
    D_\theta^k(\psi) := \frac{1}{\Delta t}\int_{t_{k-1}}^{t_k} a_\theta(\theta(s)-\theta^k,\psi)\,ds.
\end{align*}

We first eliminate the displacement derivative $\dt e_u^{h,k}$ from \eqref{eq:err-1}. Let $\boldsymbol{u}_{F,h}^k\in\boldsymbol{V}_h$ solve $a_{\boldsymbol{u}}(\boldsymbol{u}_{F,h}^k,\boldsymbol{v})=F_u^k(\boldsymbol{v})$ for all $\boldsymbol{v}\in\boldsymbol{V}_h$. By coercivity of $a_{\boldsymbol{u}}$, $\|\boldsymbol{u}_{F,h}^k\|_{\boldsymbol{V}} \le C\|F_u^k\|_{\boldsymbol{V}_h'}$. Comparing \eqref{eq:err-1} with the definition of the discrete elastic lift \eqref{eq:disc-elastic-lift} (applied to $y_h=(\dt e_p^{h,k},\dt e_\theta^{h,k})$), we have
\[
	\dt e_u^{h,k} = \boldsymbol{u}_{F,h}^k + \boldsymbol{u}_{(\dt e_p^{h,k},\dt e_\theta^{h,k}),h}, \qquad \eta^k:=(e_p^{h,k},e_\theta^{h,k})\in Q_h\times\Psi_h.
\]
%
Substituting into \eqref{eq:err-2}--\eqref{eq:err-3} and using $b_p(q,\boldsymbol{u}_{y,h})+b_\theta(\psi,\boldsymbol{u}_{y,h}) = -a_{\boldsymbol{u}}(\boldsymbol{u}_{y,h},\boldsymbol{u}_{z,h})$ with $z=(q,\psi) \in Q_h \times \Psi_h$ from \eqref{eq:disc-elastic-lift}, eliminates $\dt e_u^{h,k}$ entirely and gives
\begin{equation}
\label{eq:disc-eta-error-eq}
	m_{{\rm eff},h}(\dt\eta^k,z_h) + a_p(e_p^{h,k},q) + a_\theta(e_\theta^{h,k},\psi) = \widetilde{F}^k(z_h) \qquad \forall z_h=(q,\psi)\in Q_h\times\Psi_h,
\end{equation}
where $\widetilde{F}^k(z_h) := -F_p^k(q) - F_\theta^k(\psi) + b_p(q,u_{F,h}^k) + b_\theta(\psi,u_{F,h}^k)$. Taking $z_h=\eta^k$ in \eqref{eq:disc-eta-error-eq} and using the symmetric positive definiteness (via Lemma~\ref{lem:disc-effective-coercivity}) of $m_{{\rm eff},h}$,
\begin{align*}
	m_{{\rm eff},h}(\dt\eta^k,\eta^k) &= \frac{1}{2\Delta t}\Bigl(\|\eta^k\|_{m_{{\rm eff},h}}^2 - \|\eta^{k-1}\|_{m_{{\rm eff},h}}^2 + \|\eta^k-\eta^{k-1}\|_{m_{{\rm eff},h}}^2\Bigr) 
    \\
    &\ge \frac{1}{2\Delta t}\Bigl(\|\eta^k\|_{m_{{\rm eff},h}}^2 - \|\eta^{k-1}\|_{m_{{\rm eff},h}}^2\Bigr),
\end{align*}
where $\|\eta\|_{m_{{\rm eff},h}}^2 := m_{{\rm eff},h}(\eta,\eta)$, so that
\[
	\frac{1}{2\Delta t}\bigl(\|\eta^k\|_{m_{{\rm eff},h}}^2 - \|\eta^{k-1}\|_{m_{{\rm eff},h}}^2\bigr) + a_p(e_p^{h,k},e_p^{h,k}) + a_\theta(e_\theta^{h,k},e_\theta^{h,k}) \le \widetilde{F}^k(\eta^k).
\]
By Young's inequality and coercivity of $a_p,a_\theta$, $\widetilde{F}^k(\eta^k) \le \tfrac12\bigl(a_p(e_p^{h,k},e_p^{h,k})+a_\theta(e_\theta^{h,k},e_\theta^{h,k})\bigr) + C\|\widetilde{F}^k\|_{X_h'}^2$, where $\|\widetilde{F}^k\|_{X_h'} := \sup_{0\ne z_h\in Q_h\times\Psi_h} |\widetilde{F}^k(z_h)|/\|z_h\|_{Q\times\Psi}$. Multiplying by $2\Delta t$ and summing from $k=1$ to $m\le N$,
\[
	\|\eta^m\|_{m_{{\rm eff},h}}^2 + \Delta t\sum_{k=1}^m\bigl(\|e_p^{h,k}\|_{a_p}^2+\|e_\theta^{h,k}\|_{a_\theta}^2\bigr) \le \|\eta^0\|_{m_{{\rm eff},h}}^2 + C\Delta t\sum_{k=1}^m \|\widetilde{F}^k\|_{X_h'}^2.
\]
By Lemma~\ref{lem:disc-effective-coercivity} and boundedness of $m_{{\rm eff},h}$,
\begin{equation}
\label{eq:e-energy-3field}
\begin{aligned}
	&\max_{1\le k\le N}\bigl(\|e_p^{h,k}\|_{L^2}^2+\|e_\theta^{h,k}\|_{L^2}^2\bigr) + \Delta t\sum_{k=1}^N\bigl(\|e_p^{h,k}\|_{a_p}^2+\|e_\theta^{h,k}\|_{a_\theta}^2\bigr)
	\\
	&\le C\bigl(\|e_p^{h,0}\|_{L^2}^2+\|e_\theta^{h,0}\|_{L^2}^2\bigr) + C\Delta t\sum_{k=1}^N\|\widetilde{F}^k\|_{X_h'}^2.
\end{aligned}
\end{equation}
Since $\|u_{F,h}^k\|_{\boldsymbol{V}}\le C\|F_u^k\|_{\boldsymbol{V}_h'}$, we have $\|\widetilde{F}^k\|_{X_h'} \le C\bigl(\|F_p^k\|_{a_p,*}+\|F_\theta^k\|_{a_\theta,*}+\|F_u^k\|_{\boldsymbol{V}_h'}\bigr)$, so \eqref{eq:e-energy-3field} gives the pressure--temperature stability estimate
\begin{equation}
\label{eq:eta-final-stability}
\begin{aligned}
	&\max_{1\le k\le N}\bigl(\|e_p^{h,k}\|_{L^2}+\|e_\theta^{h,k}\|_{L^2}\bigr) + \Bigl(\Delta t\sum_{k=1}^N\bigl(\|e_p^{h,k}\|_{a_p}^2+\|e_\theta^{h,k}\|_{a_\theta}^2\bigr)\Bigr)^{1/2}
	\\
	&\le C\bigl(\|e_p^{h,0}\|_{L^2}+\|e_\theta^{h,0}\|_{L^2}\bigr) + C\Bigl(\Delta t\sum_{k=1}^N\bigl(\|F_u^k\|_{\boldsymbol{V}_h'}^2+\|F_p^k\|_{a_p,*}^2+\|F_\theta^k\|_{a_\theta,*}^2\bigr)\Bigr)^{1/2}.
\end{aligned}
\end{equation}
%

Summing the reconstruction formula $\dt e_u^{h,k} = u_{F,h}^k + \boldsymbol{u}_{\dt\eta^k,h}$ from $j=1$ to $k$ and using linearity of the discrete elastic lift, which gives the telescoping identity $\Delta t\sum_{j=1}^k \boldsymbol{u}_{\dt\eta^j,h} = \boldsymbol{u}_{\eta^k,h}-\boldsymbol{u}_{\eta^0,h}$, we obtain
\[
	e_u^{h,k} = e_u^{h,0} + \boldsymbol{u}_{\eta^k,h} - \boldsymbol{u}_{\eta^0,h} + \Delta t\sum_{j=1}^k u_{F,h}^j.
\]
By the discrete analogue of \eqref{eq:elastic-lift-bound} (coercivity of $a_{\boldsymbol{u}}$ applied to \eqref{eq:disc-elastic-lift}), $\|\boldsymbol{u}_{\eta^k,h}\|_{\boldsymbol{V}} \le C(\|e_p^{h,k}\|_{L^2}+\|e_\theta^{h,k}\|_{L^2})$, and similarly for $\boldsymbol{u}_{\eta^0,h}$; by the Cauchy--Schwarz inequality in time, $\Delta t\sum_{j=1}^k\|u_{F,h}^j\|_{\boldsymbol{V}} \le C\sqrt{t_k}\bigl(\Delta t\sum_{j=1}^N\|F_u^j\|_{\boldsymbol{V}_h'}^2\bigr)^{1/2} \le C_T\bigl(\Delta t\sum_{j=1}^N\|F_u^j\|_{\boldsymbol{V}_h'}^2\bigr)^{1/2}$. Hence, 
\begin{multline*}
    \max_{1\le k\le N}\|e_u^{h,k}\|_{\boldsymbol{V}}
	\le C_T\Bigl(\|e_u^{h,0}\|_{\boldsymbol{V}} + \|e_p^{h,0}\|_{L^2} + \|e_\theta^{h,0}\|_{L^2} + \max_{1\le k\le N}\bigl(\|e_p^{h,k}\|_{L^2}+\|e_\theta^{h,k}\|_{L^2}\bigr) 
    \\
    + \Bigl(\Delta t\sum_{k=1}^N\|F_u^k\|_{\boldsymbol{V}_h'}^2\Bigr)^{1/2}\Bigr).
\end{multline*}
Combining this with \eqref{eq:eta-final-stability} gives
\begin{equation}
\label{eq:e-final-stability}
\begin{aligned}
&\max_{1\le k\le N}\Bigl(\|e_u^{h,k}\|_{a_{\boldsymbol{u}}} + \|e_p^{h,k}\|_{L^2} + \|e_\theta^{h,k}\|_{L^2}\Bigr)
+ \Bigl(\Delta t\sum_{k=1}^N\bigl(\|e_p^{h,k}\|_{a_p}^2 + \|e_\theta^{h,k}\|_{a_\theta}^2\bigr)\Bigr)^{1/2}
\\
&\le C\bigl(\|e_u^{h,0}\|_{a_{\boldsymbol{u}}} + \|e_p^{h,0}\|_{L^2} + \|e_\theta^{h,0}\|_{L^2}\bigr)
+ C\Bigl(\Delta t\sum_{k=1}^N\bigl(\|F_u^k\|_{\boldsymbol{V}_h'}^2 + \|F_p^k\|^2_{a_p,*} + \|F_\theta^k\|^2_{a_\theta,*}\bigr)\Bigr)^{1/2}.
\end{aligned}
\end{equation}

By \eqref{eq:initial-data-assumption}, \eqref{eq:elliptic-projection-approx1}, \eqref{eq:elliptic-projection-approx2},
\[
    \|e_u^{h,0}\|_{a_{\boldsymbol{u}}} + \|e_p^{h,0}\|_{L^2} + \|e_\theta^{h,0}\|_{L^2} \le Ch^s\bigl(\|\boldsymbol{u}(0)\|_{H^{1+s}} + \|p(0)\|_{H^{1+s}} + \|\theta(0)\|_{H^{1+s}}\bigr).
\]
By continuity of $b_p$, $b_\theta$ and Korn's inequality,
\[
    \|F_u^k\|_{\boldsymbol{V}_h'} \le C\bigl(\|\dt e_p^{I,k}\|_{L^2} + \|\dt e_\theta^{I,k}\|_{L^2}\bigr).
\]
Similarly, using the $a_p$- and $a_\theta$-dual norms,
\[
    \|F_p^k\|_{a_p,*} \le C\bigl(\|\dt e_u^{I,k}\|_{H^1} + \|\dt e_p^{I,k}\|_{L^2} + \|\dt e_\theta^{I,k}\|_{L^2} + \|D_p^k\|_{a_p,*}\bigr),
\]
\[
    \|F_\theta^k\|_{a_\theta,*} \le C\bigl(\|\dt e_u^{I,k}\|_{H^1} + \|\dt e_p^{I,k}\|_{L^2} + \|\dt e_\theta^{I,k}\|_{L^2} + \|D_\theta^k\|_{a_\theta,*}\bigr).
\]
Standard estimates using \eqref{eq:elliptic-projection-approx1}, \eqref{eq:elliptic-projection-approx2}, and the definitions of $e_u^{I,k}$, $e_p^{I,k}$, $e_\theta^{I,k}$ give
\begin{multline*}
    \Delta t\sum_{k=1}^N\bigl(\|\dt e_u^{I,k}\|_{H^1}^2 + \|\dt e_p^{I,k}\|_{L^2}^2 + \|\dt e_\theta^{I,k}\|_{L^2}^2\bigr)
    \\
    \le C h^{2s}\bigl(\|\dot{\boldsymbol{u}}\|_{L^2(0,T;H^{1+s})}^2 + \|\dot{p}\|_{L^2(0,T;H^{1+s})}^2 + \|\dot{\theta}\|_{L^2(0,T;H^{1+s})}^2\bigr).
\end{multline*}
For the diffusion residuals,
\[
    \Delta t\sum_{k=1}^N\|D_p^k\|_{a_p,*}^2 \le C(\Delta t)^2\|\dot{p}\|_{L^2(0,T;H^1)}^2, \qquad
    \Delta t\sum_{k=1}^N\|D_\theta^k\|_{a_\theta,*}^2 \le C(\Delta t)^2\|\dot{\theta}\|_{L^2(0,T;H^1)}^2.
\]

For the final estimate, combining the above with \eqref{eq:e-final-stability} and the triangle inequality $\|e_u^k\|_{H^1} \le \|e_u^{I,k}\|_{H^1} + \|e_u^{h,k}\|_{H^1}$, and analogously for $e_p^k$ and $e_\theta^k$, we obtain \eqref{eq:forward-aux-estimate}. The space-time $L^2$ errors are bounded by introducing piecewise-constant-in-time interpolants $(\boldsymbol{u}_\tau, p_\tau, \theta_\tau)$ and applying the same interpolation estimates together with \eqref{eq:e-final-stability} and \eqref{eq:e-energy-3field}.
This completes the proof.
\end{proof}


\begin{lemma}
\label{lem:adjoint-error-fixed-control}
Let $(\bar{m}_p,\bar{m}_\theta)\in \mathcal{M}_{ad}$ be fixed. Let $(\boldsymbol{u},p,\theta)$ be the
corresponding continuous state, $(\boldsymbol{w},r,\phi)$ the continuous adjoint, $(\boldsymbol{u}_h,p_h,\theta_h)$ the fully discrete state, and $(\boldsymbol{w}_h,r_h,\phi_h)$ the corresponding discrete adjoint generated by the same controls.

Assume that
\[
 r, \phi \in H^1(0,T;H^{1+s}(\Omega)),
 \quad 
 \boldsymbol{w}\in H^1(0,T;\boldsymbol{H}^{1+s}(\Omega)),
\]
for some $0<s\le 1$. Then
\begin{multline}
    \|r_h-r\|_{L^2(0,T;L^2(\Omega))} + \|\phi_h-\phi\|_{L^2(0,T;L^2(\Omega))}
    \\
    \le C\left( h^s+\Delta t + \|\boldsymbol{u}_h-\boldsymbol{u}\|_{L^2(0,T;L^2(\Omega))} +\|p_h-p\|_{L^2(0,T;L^2(\Omega))} + \|\theta_h-\theta\|_{L^2(0,T;L^2(\Omega))}\right).
\end{multline}
Consequently, if the fixed-control state error satisfies
\[
\|\boldsymbol{u}_h-\boldsymbol{u}\|_{L^2(0,T;L^2(\Omega))}
+ \|p_h-p\|_{L^2(0,T;L^2(\Omega))}
+ \|\theta_h-\theta\|_{L^2(0,T;L^2(\Omega))} \le C(h^s+\Delta t),
\]
then
\[
\|r_h-r\|_{L^2(0,T;L^2(\Omega))} + \|\phi_h-\phi\|_{L^2(0,T;L^2(\Omega))} \le C(h^s+\Delta t).
\]
\end{lemma}
\begin{proof}
We introduce the auxiliary fully discrete adjoint $(\widetilde{\boldsymbol{w}}_h,\widetilde{r}_h,\widetilde{\phi}_h)$ driven by the exact continuous state $(\boldsymbol{u},p,\theta)$: its right-hand sides are $\alpha_{\boldsymbol{u},C}(\boldsymbol{u}-\boldsymbol{u}_C)$, $\alpha_{p,C}(p-p_C)$, $\alpha_{\theta,C}(\theta-\theta_C)$, instead of the discrete state values. Decompose
\[
r_h-r = (r_h-\widetilde{r}_h) + (\widetilde{r}_h-r), \qquad \phi_h-\phi = (\phi_h-\widetilde{\phi}_h) + (\widetilde{\phi}_h-\phi).
\]

\noindent\textbf{Estimate of $r_h-\widetilde{r}_h$ and $\phi_h-\widetilde{\phi}_h$.}
Setting $\varepsilon_w^k := \boldsymbol{w}_h^k-\widetilde{\boldsymbol{w}}_h^k$, $\varepsilon_r^k := r_h^k-\widetilde{r}_h^k$, $\varepsilon_\phi^k := \phi_h^k-\widetilde{\phi}_h^k$, subtracting the two discrete adjoint systems \eqref{eq:fullydiscreteadj-eqs11} gives, for $k=0,\ldots,N-1$, a three-field backward system with differences at level $k+1$ and diffusion terms at level $k$:
\begin{align}
\label{eq:adj-err-1}
-a_{\boldsymbol{u}}(\dt\varepsilon_w^{k+1},\boldsymbol{v}) - b_p(\dt\varepsilon_r^{k+1},\boldsymbol{v}) - b_\theta(\dt\varepsilon_\phi^{k+1},\boldsymbol{v})
&= \frac{\alpha_{\boldsymbol{u},C}}{\Delta t}(\boldsymbol{u}_h-\boldsymbol{u},\boldsymbol{v})_{I_k\times\Omega},
\\
\label{eq:adj-err-2}
-b_p(q,\dt\varepsilon_w^{k+1}) + (s_{pp}\dt\varepsilon_r^{k+1},q)_\Omega + (s_{p\theta}\dt\varepsilon_\phi^{k+1},q)_\Omega - a_p(\varepsilon_r^k,q)
&= \frac{\alpha_{p,C}}{\Delta t}(p_h-p,q)_{I_k\times\Omega},
\\
\label{eq:adj-err-3}
-b_\theta(\psi,\dt\varepsilon_w^{k+1}) + (s_{\theta p}\dt\varepsilon_r^{k+1},\psi)_\Omega + (s_{\theta\theta}\dt\varepsilon_\phi^{k+1},\psi)_\Omega - a_\theta(\varepsilon_\phi^k,\psi)
&= \frac{\alpha_{\theta,C}}{\Delta t}(\theta_h-\theta,\psi)_{I_k\times\Omega}.
\end{align}
Define $\widehat{\varepsilon}_w^j:=\varepsilon_w^{N-j}$, $\widehat{\varepsilon}_r^j:=\varepsilon_r^{N-j}$, $\widehat{\varepsilon}_\phi^j:=\varepsilon_\phi^{N-j}$ for $j=0,\ldots,N$; the terminal condition at $k=N$ becomes the initial condition $(\widehat{\varepsilon}_w^0,\widehat{\varepsilon}_r^0,\widehat{\varepsilon}_\phi^0)=(0,0,0)$. For the equation on $I_k$, set $j=N-k$, so $k=N-j$ and $k+1=N-(j-1)$; then
\[
	\varepsilon^k = \widehat{\varepsilon}^j, \qquad
	\dt\varepsilon^{k+1} = \frac{\varepsilon^{k+1}-\varepsilon^k}{\Delta t} = \frac{\widehat{\varepsilon}^{j-1}-\widehat{\varepsilon}^j}{\Delta t} = -\dt\widehat{\varepsilon}^j.
\]
Substituting these two relations into \eqref{eq:adj-err-1}--\eqref{eq:adj-err-3} converts the negative backward-difference terms into forward differences, 
while the diffusion terms and right-hand sides are unchanged except for the interval relabeling $I_k=I_{N-j}=:\widehat{I}_j$; the diffusion terms are evaluated at $\widehat{\varepsilon}^j$, so the reindexed system has the same left-hand side structure as \eqref{eq:err-1}--\eqref{eq:err-3}. This gives, for $j=1,\ldots,N$,
\begin{align*}
a_{\boldsymbol{u}}(\dt\widehat{\varepsilon}_w^j,\boldsymbol{v}) + b_p(\dt\widehat{\varepsilon}_r^j,\boldsymbol{v}) + b_\theta(\dt\widehat{\varepsilon}_\phi^j,\boldsymbol{v})
&= \frac{\alpha_{\boldsymbol{u},C}}{\Delta t}(\boldsymbol{u}_h-\boldsymbol{u},\boldsymbol{v})_{\widehat{I}_j\times\Omega},
\\
b_p(q,\dt\widehat{\varepsilon}_w^j) - (s_{pp}\dt\widehat{\varepsilon}_r^j,q)_\Omega - (s_{p\theta}\dt\widehat{\varepsilon}_\phi^j,q)_\Omega - a_p(\widehat{\varepsilon}_r^j,q)
&= \frac{\alpha_{p,C}}{\Delta t}(p_h-p,q)_{\widehat{I}_j\times\Omega},
\\
b_\theta(\psi,\dt\widehat{\varepsilon}_w^j) - (s_{\theta p}\dt\widehat{\varepsilon}_r^j,\psi)_\Omega - (s_{\theta\theta}\dt\widehat{\varepsilon}_\phi^j,\psi)_\Omega - a_\theta(\widehat{\varepsilon}_\phi^j,\psi)
&= \frac{\alpha_{\theta,C}}{\Delta t}(\theta_h-\theta,\psi)_{\widehat{I}_j\times\Omega},
\end{align*}
where $\widehat{I}_j$ denotes the time interval corresponding to $I_{N-j}$ under the reindexing. Writing the free index as $k$ again, since $j$ ranges over $1,\ldots,N$ exactly as $k$ does in \eqref{eq:err-1}--\eqref{eq:err-3}, this is exactly the forward error system \eqref{eq:err-1}--\eqref{eq:err-3}, with the state errors $\boldsymbol{u}_h-\boldsymbol{u}$, $p_h-p$, $\theta_h-\theta$ in place of the residuals $F_u^k,F_p^k,F_\theta^k$, and zero initial data $(\widehat{\varepsilon}_w^0,\widehat{\varepsilon}_r^0,\widehat{\varepsilon}_\phi^0)=(0,0,0)$. For the displacement residual we use the $\boldsymbol{V}_h'$-norm, while for the pressure and temperature residuals we use the $a_p$- and $a_\theta$-dual norms matching the norms in which $F_p^k,F_\theta^k$ are measured in \eqref{eq:e-final-stability}. In each case, since the residual has the form $\frac{1}{\Delta t}(\varphi,\cdot)_{\widehat{I}_j\times\Omega}$ for a fixed function $\varphi$, namely $\boldsymbol{u}_h-\boldsymbol{u}$, $p_h-p$, or $\theta_h-\theta$, the Cauchy--Schwarz inequality in time gives, with $\|\cdot\|_*$ denoting the corresponding dual norms $\boldsymbol{V}_h'$, $a_p$-dual, or $a_\theta$-dual,
\[
	\left\|\frac{1}{\Delta t}(\varphi,\cdot)_{\widehat{I}_j\times\Omega}\right\|_* \le \frac{1}{\Delta t}\|\varphi\|_{L^2(\widehat{I}_j;L^2(\Omega))}\cdot\sqrt{\Delta t} = \Delta t^{-1/2}\|\varphi\|_{L^2(\widehat{I}_j;L^2(\Omega))},
\]
so that
\[
	\Delta t\sum_{j=1}^N \left\|\frac{1}{\Delta t}(\varphi,\cdot)_{\widehat{I}_j\times\Omega}\right\|_*^2 \le C\sum_{j=1}^N \|\varphi\|_{L^2(\widehat{I}_j;L^2(\Omega))}^2 = C\|\varphi\|_{L^2(0,T;L^2(\Omega))}^2.
\]
Applying this scaling to each of the three residuals 
and then the three-field stability estimate \eqref{eq:e-final-stability} 
gives
\begin{multline}
    \|r_h-\widetilde{r}_h\|_{L^2(0,T;L^2(\Omega))} + \|\phi_h-\widetilde{\phi}_h\|_{L^2(0,T;L^2(\Omega))}
    \\
    \le C\bigl(\|\boldsymbol{u}_h-\boldsymbol{u}\|_{L^2L^2} + \|p_h-p\|_{L^2L^2} + \|\theta_h-\theta\|_{L^2L^2}\bigr).
\end{multline}

\noindent\textbf{Estimate of $\widetilde{r}_h-r$ and $\widetilde{\phi}_h-\phi$.}
$(\widetilde{\boldsymbol{w}}_h,\widetilde{r}_h,\widetilde{\phi}_h)$ is the fully discrete approximation of the continuous adjoint $(\boldsymbol{w},r,\phi)$ with exact state-data right-hand sides. The continuous and discrete adjoints both satisfy zero terminal data, so under the time reversal $s=T-t$ the terminal condition becomes a zero initial condition, and the analogue of the initial-error term $\|e^{h,0}\|$ in Lemma~\ref{lemma:aux-forward-error} vanishes identically. The dG(0) three-field error estimate (backward-in-time version of Lemma~\ref{lemma:aux-forward-error}, applied with this zero initial data) gives
\[
\|\widetilde{r}_h-r\|_{L^2(0,T;L^2(\Omega))} + \|\widetilde{\phi}_h-\phi\|_{L^2(0,T;L^2(\Omega))} \le C(h^s+\Delta t),
\]
under the stated regularity of $(\boldsymbol{w},r,\phi)$.

The triangle inequality and the two steps above yield the stated estimate. The final bound $C(h^s+\Delta t)$ follows from applying Lemma~\ref{lemma:aux-forward-error} to bound the state errors.
\end{proof}
\subsection{Estimates for variational discretization of controls}
This subsection derives error estimates for the control pair $(m_p,m_{\theta})$ within the variational discretization setting.
\begin{lemma}\label{Lemma:5.7}
    Recall the reduced discrete functional $j_h(m_p,m_\theta)$ defined above, which can be written explicitly as
\[
\begin{aligned}
        j_h(m_p,m_\theta)
        =
        &\frac{\alpha_{\boldsymbol{u},C}}{2}
        \|\boldsymbol{u}_h(m_p,m_\theta)-\boldsymbol{u}_C\|_{L^2(0,T; L^2(\Omega))}^{2}
        \\
        &+
        \frac{\alpha_{p,C}}{2}
        \|p_h(m_p,m_\theta)-p_C\|_{L^2(0,T; L^2(\Omega))}^{2}  
        \\
        &+
        \frac{\alpha_{\theta,C}}{2}
        \|\theta_h(m_p,m_\theta)-\theta_C\|_{L^2(0,T; L^2(\Omega))}^{2}
        \\
        &+
        \frac{\gamma_p}{2}
        \|m_p\|_{L^2(0,T; L^2(\Omega))}^{2}
        +
        \frac{\gamma_\theta}{2}
        \|m_\theta\|_{L^2(0,T; L^2(\Omega))}^{2} .
\end{aligned}
\]
    For $(m_{p,1},m_{\theta,1}), (m_{p,2},m_{\theta,2}) \in \mc{M}_{ad}$ the following estimate holds:
	\begin{align*}
		\bigl(j_{h}'(m_{p,1},m_{\theta,1}) - j_{h}'(m_{p,2},m_{\theta,2}),(m_{p,1}-m_{p,2},m_{\theta,1}-m_{\theta,2})\bigr)_{(0,T)\times \Omega} 
        \\
        \ge \gamma_p \|m_{p,1}-m_{p,2}\|_{L^2(0,T;L^2(\Omega))}^2 + \gamma_\theta\|m_{\theta,1}-m_{\theta,2}\|_{L^2(0,T;L^2(\Omega))}^2,
	\end{align*}
\end{lemma}
\begin{proof}
Since \eqref{eq:fullydiscrete-eqs} is linear in the controls, the discrete state depends affinely on $(m_p,m_\theta)$. Set $\delta m_p:=m_{p,1}-m_{p,2}$, $\delta m_\theta:=m_{\theta,1}-m_{\theta,2}$, $\delta\boldsymbol{u}_h:=\boldsymbol{u}_h(m_{p,1},m_{\theta,1})-\boldsymbol{u}_h(m_{p,2},m_{\theta,2})$, $\delta p_h:=p_h(m_{p,1},m_{\theta,1})-p_h(m_{p,2},m_{\theta,2})$, $\delta\theta_h:=\theta_h(m_{p,1},m_{\theta,1})-\theta_h(m_{p,2},m_{\theta,2})$. The second variation of $j_h$ gives
\begin{align*}
&\bigl(j_h'(m_{p,1},m_{\theta,1})-j_h'(m_{p,2},m_{\theta,2}),(\delta m_p,\delta m_\theta)\bigr)_{(0,T) \times \Omega}
\\
&=
\alpha_{\boldsymbol{u},C}\|\delta\boldsymbol{u}_h\|_{L^2(0,T;L^2(\Omega))}^{2}
+
\alpha_{p,C}\|\delta p_h\|_{L^2(0,T;L^2(\Omega))}^{2}
+
\alpha_{\theta,C}\|\delta\theta_h\|_{L^2(0,T;L^2(\Omega))}^{2}
\\
&\quad + \gamma_p\|\delta m_p\|_{L^2(0,T;L^2(\Omega))}^{2}
+
\gamma_\theta\|\delta m_\theta\|_{L^2(0,T;L^2(\Omega))}^{2}.
\end{align*}
Since all $\alpha$ coefficients are non-negative and $\gamma_p, \gamma_\theta > 0$, the result follows.
\end{proof}

\begin{theorem}\label{Theorem:5.8}
	Let $(\boldsymbol{u},p,\theta,\bar{m}_p,\bar{m}_\theta)$ and $(\boldsymbol{u}_h,p_h,\theta_h,m_{p,h},m_{\theta,h})$ denote the continuous and fully discrete optimal solutions. Suppose that the continuous state and adjoint satisfy
	\begin{align}\label{eq:adjoint-regularity}
		p, r, \theta, \phi \in H^1(0,T; H^{1+s}(\Omega)) ,
		\qquad
		\boldsymbol{u}, \boldsymbol{w} \in H^1(0,T; \boldsymbol{H}^{1+s}(\Omega)),
	\end{align}
    and the initial data assumption \eqref{eq:initial-data-assumption} holds. 
	Then, we have the estimates
    \begin{align*}
		\|\bar{m}_p-m_{p,h}\|_{L^2(0,T;L^2(\Omega))} + \|\bar{m}_\theta-m_{\theta,h}\|_{L^2(0,T;L^2(\Omega))} &\le C (\Delta t + h^s)
	\end{align*}
    with $C>0$ depending on the solution regularities in \eqref{eq:adjoint-regularity}.
\end{theorem}
\begin{proof}
Let $(\bar{m}_p,\bar{m}_\theta)$ and $(m_{p,h},m_{\theta,h})$ denote the continuous and discrete optimal controls, respectively. We treat the two controls separately but use the same argument for each. For the fluid control, with the convention that $j'(\bar{m}_p,\bar{m}_\theta)_{m_p} = \gamma_p \bar{m}_p + r$ and $j_h'(m_{p,h},m_{\theta,h})_{m_p} = \gamma_p m_{p,h} + r_h(m_{p,h},m_{\theta,h})$, and similarly for the heat control:
\[
        j'(\bar{m}_p,\bar{m}_\theta)_{m_\theta}=\gamma_\theta \bar{m}_\theta+\phi,
        \qquad
        j_h'(m_{p,h},m_{\theta,h})_{m_\theta}=\gamma_\theta m_{\theta,h}+\phi_h(m_{p,h},m_{\theta,h}).
\]
Below, $r_h(\bar{m}_p,\bar{m}_\theta)$ and $\phi_h(\bar{m}_p,\bar{m}_\theta)$ denote the discrete adjoint variables computed from the discrete state generated by the \emph{continuous} optimal controls $(\bar{m}_p,\bar{m}_\theta)$ (as distinct from $r_h(m_{p,h},m_{\theta,h})$, $\phi_h(m_{p,h},m_{\theta,h})$ above, which are generated by the discrete optimal controls).

By strong monotonicity (Lemma~\ref{Lemma:5.7}), 
\begin{multline*}
    \gamma_p \|\bar{m}_p-m_{p,h}\|_{L^2(0,T;L^2(\Omega))}^2 + \gamma_\theta\|\bar{m}_\theta-m_{\theta,h}\|_{L^2(0,T;L^2(\Omega))}^2
    \\
    \le \bigl(j_h'(\bar{m}_p,\bar{m}_\theta)-j_h'(m_{p,h},m_{\theta,h}),(\bar{m}_p-m_{p,h},\bar{m}_\theta-m_{\theta,h})\bigr)_{(0,T) \times \Omega}.
\end{multline*}
Expanding and using the discrete variational inequalities \eqref{eq:varineq-g}--\eqref{eq:varineq-m} with $(\tilde{m}_p,\tilde{m}_\theta)=(\bar{m}_p,\bar{m}_\theta)$, and the continuous variational inequalities \eqref{ctsvi-g}--\eqref{ctsvi-m} with $(m_p,m_\theta)=(m_{p,h},m_{\theta,h})$, we obtain
\begin{multline*}
    \gamma_p \|\bar{m}_p-m_{p,h}\|_{L^2(0,T;L^2(\Omega))}^2 + \gamma_\theta\|\bar{m}_\theta-m_{\theta,h}\|_{L^2(0,T;L^2(\Omega))}^2
    \\
    \le \bigl(r_h(\bar{m}_p,\bar{m}_\theta)-r,\,\bar{m}_p-m_{p,h}\bigr)_{(0,T) \times \Omega} + \bigl(\phi_h(\bar{m}_p,\bar{m}_\theta)-\phi,\bar{m}_\theta-m_{\theta,h}\bigr)_{(0,T) \times \Omega}.
\end{multline*}
By the Cauchy--Schwarz inequality and Lemma~\ref{lem:adjoint-error-fixed-control},
\begin{multline*}
    \gamma_p \|\bar{m}_p-m_{p,h}\|_{L^2(0,T;L^2(\Omega))}^2 + \gamma_\theta\|\bar{m}_\theta-m_{\theta,h}\|_{L^2(0,T;L^2(\Omega))}^2
    \\
    \le C(h^s+\Delta t)(\|\bar{m}_p-m_{p,h}\|_{L^2(0,T;L^2(\Omega))} + \|\bar{m}_\theta-m_{\theta,h}\|_{L^2(0,T;L^2(\Omega))}).
\end{multline*}
Dividing by $(\|\bar{m}_p-m_{p,h}\|_{L^2(0,T;L^2(\Omega))}^2 + \|\bar{m}_\theta-m_{\theta,h}\|_{L^2(0,T;L^2(\Omega))}^2)^{1/2}$ (when nonzero) yields the result.
This completes the proof.
\end{proof}
\begin{theorem}\label{Theorem:5.9}
Let $(\boldsymbol u,p,\theta,\bar{m}_p,\bar{m}_\theta)$ be the continuous optimal solution and
let $(\boldsymbol{u}_h,p_h,\theta_h,m_{p,h},m_{\theta,h})$ be the fully discrete optimal solution. Under the
regularity assumptions of Lemma~\ref{lemma:aux-forward-error},
Lemma~\ref{lem:adjoint-error-fixed-control}, and
Theorem~\ref{Theorem:5.8}, there exists a constant $C>0$, independent of
$h$ and $\Delta t$, such that
\[
\begin{aligned}
&\max_{1\le k\le N}
\left(
\|\boldsymbol u^k-\boldsymbol{u}_h^k\|_{H^1(\Omega)}
+
\|p^k-p_h^k\|_{L^2(\Omega)}
+
\|\theta^k-\theta_h^k\|_{L^2(\Omega)}
\right)
\\
&\qquad
+
\|\boldsymbol{u}-\boldsymbol{u}_h\|_{L^2(0,T;H^1(\Omega))}
+
\|p-p_h\|_{L^2(0,T;L^2(\Omega))}
+
\|\theta-\theta_h\|_{L^2(0,T;L^2(\Omega))}
\\
&\qquad \le
C(h^s+\Delta t).
\end{aligned}
\]
In particular, if $s=1$, then
\[
\begin{aligned}
&\max_{1\le k\le N}
\left(
\|\boldsymbol u^k-\boldsymbol{u}_h^k\|_{H^1(\Omega)}
+
\|p^k-p_h^k\|_{L^2(\Omega)}
+
\|\theta^k-\theta_h^k\|_{L^2(\Omega)}
\right)
\\
&\qquad
+
\|\boldsymbol{u}-\boldsymbol{u}_h\|_{L^2(0,T;H^1(\Omega))}
+
\|p-p_h\|_{L^2(0,T;L^2(\Omega))}
+
\|\theta-\theta_h\|_{L^2(0,T;L^2(\Omega))}
\\
&\qquad \le
C(h+\Delta t).
\end{aligned}
\]
\end{theorem}
\begin{proof}
Let $(\boldsymbol{u}_h(\bar{m}_p,\bar{m}_\theta), p_h(\bar{m}_p,\bar{m}_\theta), \theta_h(\bar{m}_p,\bar{m}_\theta))$ denote the fully discrete state generated by the continuous optimal controls $(\bar{m}_p,\bar{m}_\theta)$. Decompose:
\begin{align*}
\boldsymbol{u}^k - \boldsymbol{u}_h^k &= (\boldsymbol{u}^k - \boldsymbol{u}_h^k(\bar{m}_p,\bar{m}_\theta)) + (\boldsymbol{u}_h^k(\bar{m}_p,\bar{m}_\theta) - \boldsymbol{u}_h^k), \\
p^k - p_h^k &= (p^k - p_h^k(\bar{m}_p,\bar{m}_\theta)) + (p_h^k(\bar{m}_p,\bar{m}_\theta) - p_h^k), \\
\theta^k - \theta_h^k &= (\theta^k - \theta_h^k(\bar{m}_p,\bar{m}_\theta)) + (\theta_h^k(\bar{m}_p,\bar{m}_\theta) - \theta_h^k).
\end{align*}
The first terms are fixed-control discretization errors bounded by Lemma~\ref{lemma:aux-forward-error}:
\begin{multline*}
    \max_{1\le k\le N}\bigl(\|e_u^k\|_{H^1} + \|e_p^k\|_{L^2} + \|e_\theta^k\|_{L^2}\bigr)
    \\ 
    + \|e_u\|_{L^2(0,T;H^1(\Omega))} + \|e_p\|_{L^2(0,T;L^2(\Omega))} + \|e_\theta\|_{L^2(0,T;L^2(\Omega))} \le C(h^s+\Delta t).
\end{multline*}

For the perturbation errors, set $\delta m_p := \bar{m}_p-m_{p,h}$, $\delta m_\theta := \bar{m}_\theta-m_{\theta,h}$, and
\begin{align*}
    \zeta_u^k &:= \boldsymbol{u}_h^k(\bar{m}_p,\bar{m}_\theta) - \boldsymbol{u}_h^k(m_{p,h},m_{\theta,h}), \quad
    \zeta_p^k := p_h^k(\bar{m}_p,\bar{m}_\theta) - p_h^k(m_{p,h},m_{\theta,h}), 
    \\
    \zeta_\theta^k &:= \theta_h^k(\bar{m}_p,\bar{m}_\theta) - \theta_h^k(m_{p,h},m_{\theta,h}).    
\end{align*}
Subtracting the discrete state equations \eqref{eq:fullydiscrete-eqs} for $(\bar{m}_p,\bar{m}_\theta)$ from those for $(m_{p,h},m_{\theta,h})$ gives, for $k=0,\ldots,N-1$ and all $(\boldsymbol{v},q,\psi)\in\boldsymbol{V}_h\times Q_h\times\Psi_h$ (matching the index convention of \eqref{eq:fullydiscrete-eqs} itself, with differences at level $k+1$ and diffusion terms at level $k+1$, since both discrete states satisfy the same dG(0) scheme),
\begin{align*}
a_{\boldsymbol{u}}(\dt\zeta_u^{k+1},\boldsymbol{v}) + b_p(\dt\zeta_p^{k+1},\boldsymbol{v}) + b_\theta(\dt\zeta_\theta^{k+1},\boldsymbol{v}) &= 0, \\
b_p(q,\dt\zeta_u^{k+1}) - (s_{pp}\dt\zeta_p^{k+1},q)_\Omega - (s_{p\theta}\dt\zeta_\theta^{k+1},q)_\Omega - a_p(\zeta_p^{k+1},q) &= \tfrac{1}{\Delta t}(\delta m_p, q)_{I_k\times\Omega}, \\
b_\theta(\psi,\dt\zeta_u^{k+1}) - (s_{\theta p}\dt\zeta_p^{k+1},\psi)_\Omega - (s_{\theta\theta}\dt\zeta_\theta^{k+1},\psi)_\Omega - a_\theta(\zeta_\theta^{k+1},\psi) &= \tfrac{1}{\Delta t}(\delta m_\theta, \psi)_{I_k\times\Omega},
\end{align*}
with $(\zeta_u^0,\zeta_p^0,\zeta_\theta^0) = (0,0,0)$ (since both discrete states share the same numerical initial data $(\boldsymbol{u}_h^0,p_h^0,\theta_h^0)$). This is exactly the fixed-control error system \eqref{eq:err-1}--\eqref{eq:err-3} after the same relabeling $k\to k+1$ used to pass from \eqref{eq:fullydiscrete-eqs} to \eqref{eq:err-1}--\eqref{eq:err-3} (equivalently, writing the free index as $k$ again and reading $\zeta^k$ in place of $e^{h,k}$), with residuals $\frac{1}{\Delta t}(\delta m_p,\cdot)_{I_k\times\Omega}$ and $\frac{1}{\Delta t}(\delta m_\theta,\cdot)_{I_k\times\Omega}$ in place of $F_p^k,F_\theta^k$, and zero residual in place of $F_u^k$. By the same Cauchy--Schwarz-in-time argument used in the proof of Lemma~\ref{lem:adjoint-error-fixed-control}, their dual norms satisfy $\Delta t\sum_k\|\Delta t^{-1}(\delta m_p,\cdot)_{I_k\times\Omega}\|_{a_p,*}^2 \le C\|\delta m_p\|_{L^2L^2}^2$ and $\Delta t\sum_k\|\Delta t^{-1}(\delta m_\theta,\cdot)_{I_k\times\Omega}\|_{a_\theta,*}^2 \le C\|\delta m_\theta\|_{L^2L^2}^2$. Applying the same effective-mass stability estimate \eqref{eq:e-final-stability} used for the fixed-control forward error, with residuals bounded as above by $\|\delta m_p\|_{L^2L^2}$ and $\|\delta m_\theta\|_{L^2L^2}$, gives
\begin{multline*}
    \max_{k}\bigl(\|\zeta_u^k\|_{H^1} + \|\zeta_p^k\|_{L^2} + \|\zeta_\theta^k\|_{L^2}\bigr)
+ \|\zeta_u\|_{L^2H^1} + \|\zeta_p\|_{L^2L^2} + \|\zeta_\theta\|_{L^2L^2} 
    \\ 
    \le C\bigl(\|\delta m_p\|_{L^2L^2} + \|\delta m_\theta\|_{L^2L^2}\bigr).
\end{multline*}
By Theorem~\ref{Theorem:5.8}, $\|\delta m_p\|_{L^2L^2} + \|\delta m_\theta\|_{L^2L^2} \le C(h^s+\Delta t)$. The result follows by the triangle inequality.
\end{proof}


\section{Numerical experiments}\label{Numerical Experiments}

The numerical experiments in this section validate the three-field thermo-poroelastic optimality system with two distributed controls.  We use the same finite element spaces as in Section~\ref{discrete formulation}: the lowest order Taylor--Hood pair $[\mathbb P_2]^2\times\mathbb P_1$ for $(\boldsymbol u,p)$ and a continuous $\mathbb P_1$ space for $\theta$.  The controls are treated by variational discretization and are evaluated through the projection formulas involving the discrete adjoint variables $r_h^k$ and $\phi_h^k$ on each time interval $I_k$.

\subsection*{Experiment 1: Full three-field manufactured convergence test}
On $\Omega=(0,1)^2$ and $T=1$, we use the boundary bubble and time factor
\[
    B(x,y)=x^3(1-x)^3y^3(1-y)^3,\qquad
    \eta(t)=t^2(T-t)^3.
\]
For the mechanical boundary partition we take
\[
    \Gamma_d=\{0\}\times[0,1], \qquad \Gamma_t=\partial\Omega\setminus\Gamma_d,
\]
so that both $|\Gamma_d|>0$ and $|\Gamma_t|>0$.  For the pressure and temperature variables we take $\Gamma_p=\Gamma_\theta=\partial\Omega$.  Since $B=0$ and $\nabla B=0$ on $\partial\Omega$, the manufactured fields below satisfy the homogeneous pressure and temperature boundary conditions, the displacement condition on $\Gamma_d$, and the homogeneous traction condition on $\Gamma_t$.

The exact state is chosen as
\[
    p^*(x,y,t)=B(x,y)\eta(t), \qquad
    \theta^*(x,y,t)=B(x,y)(1+x+y)\eta(t),
\]
\[
    \boldsymbol{u}^*(x,y,t) = B(x,y)\eta(t)
    \begin{pmatrix}
    1+x \\
    1+y
    \end{pmatrix}.
\]
The body force is manufactured from the mechanical equation only.  With $\boldsymbol{\sigma}(\boldsymbol v)=2\mu\epsilon(\boldsymbol v)+\lambda(\operatorname{div}\boldsymbol v)\boldsymbol I$, we set
\[
\boldsymbol f^*
=-\operatorname{div}\boldsymbol{\sigma}(\partial_t\boldsymbol u^*)
  +\alpha_p\nabla \partial_t p^*
  +\alpha_\theta\nabla \partial_t\theta^*.
\]
No additional scalar manufactured sources are introduced.  Instead, the exact controls are defined by the scalar state residuals,
\[
\begin{aligned}
m_p^*&=-\alpha_p\operatorname{div}(\partial_t\boldsymbol u^*)
       -s_{pp}\partial_t p^*-s_{p\theta}\partial_t\theta^*
       +\operatorname{div}(\kappa_p\nabla p^*),\\
m_\theta^*&=-\alpha_\theta\operatorname{div}(\partial_t\boldsymbol u^*)
       -s_{\theta p}\partial_t p^*-s_{\theta\theta}\partial_t\theta^*
       +\operatorname{div}(\kappa_\theta\nabla\theta^*).
\end{aligned}
\]
Thus the scalar equations have exactly the same right-hand sides as the OCP analyzed in the paper,
\[
\begin{aligned}
s_{pp}\partial_t p^*+\alpha_p\operatorname{div}(\partial_t\boldsymbol u^*)
  +s_{p\theta}\partial_t\theta^* -\operatorname{div}(\kappa_p\nabla p^*) &=-m_p^*,\\
s_{\theta\theta}\partial_t\theta^*+\alpha_\theta\operatorname{div}(\partial_t\boldsymbol u^*)
  +s_{\theta p}\partial_t p^* -\operatorname{div}(\kappa_\theta\nabla\theta^*) &=-m_\theta^*.
\end{aligned}
\]
For the optimality-system test, the adjoint variables are constructed from the inactive optimality condition.  We choose
\[
    \boldsymbol{w}^*(x,y,t)=(T-t)^2B(x,y)
    \begin{pmatrix}
    1+x \\
    1+y
    \end{pmatrix},
\]
then define the scalar adjoints by
\[
    r^*(x,y,t)=-\gamma_p m_p^*(x,y,t),\qquad
    \phi^*(x,y,t)=-\gamma_\theta m_\theta^*(x,y,t).
\]
Because $\eta(T)=\eta'(T)=0$ and the cubic bubble has vanishing second traces on $\partial\Omega$, these adjoint fields satisfy the homogeneous terminal and scalar boundary conditions.  By construction, the unconstrained optimality equations
\[
    \gamma_p m_p^*+r^*=0,\qquad
    \gamma_\theta m_\theta^*+\phi^*=0
\]
are satisfied exactly.  If box constraints are imposed, the same formula is interpreted through pointwise projection,
\[
    m_p^*=\mathcal P_{[m_{p,a},m_{p,b}]}
       \!\left(-\gamma_p^{-1}r^*\right),\qquad
    m_\theta^*=\mathcal P_{[m_{\theta,a},m_{\theta,b}]}
       \!\left(-\gamma_\theta^{-1}\phi^*\right).
\]
For the manufactured convergence tests, the bounds are chosen sufficiently wide so that the exact controls are inactive. Hence the projection formula
reduces to
\[
m_p^*=-\gamma_p^{-1}r^*,\qquad
m_\theta^*=-\gamma_\theta^{-1}\phi^*.
\]
The active-constraint experiment below is used only to visualize the
pointwise projection mechanism and is not interpreted as a manufactured convergence test.

In the implementation the controls are not assigned independent finite element spaces.  They are evaluated from the discrete projection formula using the computed adjoint variables,
\[
    m_{p,h}|_{I_k}=\mathcal P_{\mathcal M_{ad}^p}\!\left(-\gamma_p^{-1}r_h^k\right),
    \qquad
    m_{\theta,h}|_{I_k}=\mathcal P_{\mathcal M_{ad}^\theta}\!\left(-\gamma_\theta^{-1}\phi_h^k\right),
\]
which is the variational-discretization realization of \eqref{proj}.

The tracking targets are also manufactured from the continuous adjoint residual.  Writing the tracking weights as $\omega_u,\omega_p,\omega_\theta$ (all equal to one in the reported runs), define residuals $\boldsymbol R_u^*,R_p^*,R_\theta^*$ by
\[
\begin{aligned}
\boldsymbol R_u^*
&=-(-\operatorname{div}\boldsymbol{\sigma}(\partial_t\boldsymbol w^*)
  +\alpha_p\nabla\partial_t r^*
  +\alpha_\theta\nabla\partial_t\phi^*),
  \\
R_p^*
&=\alpha_p\operatorname{div}(\partial_t\boldsymbol w^*)
  +s_{pp}\partial_t r^*+s_{p\theta}\partial_t\phi^*
  -\operatorname{div}(\kappa_p\nabla r^*),
  \\
R_\theta^*
&=\alpha_\theta\operatorname{div}(\partial_t\boldsymbol w^*)
  +s_{\theta p}\partial_t r^*+s_{\theta\theta}\partial_t\phi^*
  -\operatorname{div}(\kappa_\theta\nabla\phi^*).
\end{aligned}
\]
The desired targets are then chosen as
\[
    \boldsymbol u_C=\boldsymbol u^* -\omega_u^{-1}\boldsymbol R_u^*,\qquad
    p_C=p^* -\omega_p^{-1}R_p^*,\qquad
    \theta_C=\theta^* -\omega_\theta^{-1}R_\theta^*.
\]
With these definitions, the tuple
\[
(\boldsymbol u^*,p^*,\theta^*,\boldsymbol w^*,r^*,\phi^*,m_p^*,m_\theta^*)
\]
satisfies the continuous state equation, adjoint equation, and unconstrained first-order optimality condition of the same OCP as in the analysis.  The numerical test therefore measures the discretization error of the fully discrete variational-discretization scheme without adding control-independent scalar source terms.

Table~\ref{tab:3field-mms-conv} reports relative space-time errors for spatial refinement with the default Taylor--Hood-type spaces $[\mathbb P_2]^2\times\mathbb P_1\times\mathbb P_1$ for $(\boldsymbol u,p,\theta)$.  We use the continuous manufactured targets, Gaussian quadrature in time, and $\Delta t=1/4096$ so that the temporal error is small compared with the spatial error.  The rates are computed by halving $h$.  All variables show approximately second-order spatial convergence except that the finest-mesh rate for $\boldsymbol w$ remains more sensitive to the time discretization.

\begin{table}[htbp]
\centering
\scriptsize
\caption{Spatial refinement for the boundary-bubble manufactured optimality system with $\Delta t=1/4096$, continuous manufactured targets, Gaussian time quadrature, and $[\mathbb P_2]^2\times\mathbb P_1\times\mathbb P_1$ elements. Entries are relative errors; numbers in parentheses are rates from the previous mesh.}
\label{tab:3field-mms-conv}
\begin{tabular}{c|cccc}
\hline
$N$ & $\boldsymbol u$ & $p$ & $\theta$ & $\boldsymbol w$ \\
\hline
4  & $1.542\times10^{-1}$ (~~--~~) & $2.099\times10^{-1}$ (~~--~~) & $2.125\times10^{-1}$ (~~--~~) & $1.551\times10^{-1}$ (~~--~~) \\
8  & $4.297\times10^{-2}$ (1.84) & $6.374\times10^{-2}$ (1.72) & $6.499\times10^{-2}$ (1.71) & $4.327\times10^{-2}$ (1.84) \\
16 & $1.095\times10^{-2}$ (1.97) & $1.683\times10^{-2}$ (1.92) & $1.720\times10^{-2}$ (1.92) & $1.105\times10^{-2}$ (1.97) \\
32 & $2.754\times10^{-3}$ (1.99) & $4.276\times10^{-3}$ (1.98) & $4.365\times10^{-3}$ (1.98) & $2.840\times10^{-3}$ (1.96) \\
64 & $7.221\times10^{-4}$ (1.93) & $1.085\times10^{-3}$ (1.98) & $1.098\times10^{-3}$ (1.99) & $9.478\times10^{-4}$ (1.58) \\
\hline\hline
$N$ & $r$ & $\phi$ & $m_p$ & $m_\theta$ \\
\hline
4  & $1.944\times10^{-1}$ (~~--~~) & $1.988\times10^{-1}$ (~~--~~) & $1.932\times10^{-1}$ (~~--~~) & $1.976\times10^{-1}$ (~~--~~) \\
8  & $5.815\times10^{-2}$ (1.74) & $5.999\times10^{-2}$ (1.73) & $5.813\times10^{-2}$ (1.73) & $5.997\times10^{-2}$ (1.72) \\
16 & $1.513\times10^{-2}$ (1.94) & $1.567\times10^{-2}$ (1.94) & $1.513\times10^{-2}$ (1.94) & $1.567\times10^{-2}$ (1.94) \\
32 & $3.707\times10^{-3}$ (2.03) & $3.845\times10^{-3}$ (2.03) & $3.707\times10^{-3}$ (2.03) & $3.845\times10^{-3}$ (2.03) \\
64 & $8.695\times10^{-4}$ (2.09) & $8.911\times10^{-4}$ (2.11) & $8.695\times10^{-4}$ (2.09) & $8.911\times10^{-4}$ (2.11) \\
\hline
\end{tabular}
\end{table}

Table~\ref{tab:3field-mms-dt} reports temporal refinement for the same $[\mathbb P_2]^2\times\mathbb P_1\times\mathbb P_1$ elements on a fixed $128\times128$ mesh, using the continuous adjoint target construction.  The displacement, adjoints, and controls show first-order temporal convergence.  The pressure and temperature rates improve relative to coarser meshes, but the finest refinements still show some remaining spatial-error influence, especially for $\theta$.

\begin{table}[htbp]
\centering
\scriptsize
\caption{Temporal refinement for the boundary-bubble manufactured optimality system with fixed $N=128$ and $[\mathbb P_2]^2\times\mathbb P_1\times\mathbb P_1$ elements. Entries are relative errors; numbers in parentheses are rates from the previous time step.}
\label{tab:3field-mms-dt}
\begin{tabular}{c|cccc}
\hline
$n$ & $\boldsymbol u$ & $p$ & $\theta$ & $\boldsymbol w$ \\
\hline
16  & $1.325\times10^{-2}$ (~~--~~) & $4.023\times10^{-3}$ (~~--~~) & $2.164\times10^{-3}$ (~~--~~) & $7.746\times10^{-2}$ (~~--~~) \\
32  & $6.451\times10^{-3}$ (1.04) & $2.131\times10^{-3}$ (0.92) & $1.170\times10^{-3}$ (0.89) & $3.930\times10^{-2}$ (0.98) \\
64  & $3.184\times10^{-3}$ (1.02) & $1.162\times10^{-3}$ (0.87) & $6.870\times10^{-4}$ (0.77) & $1.980\times10^{-2}$ (0.99) \\
128 & $1.588\times10^{-3}$ (1.00) & $6.819\times10^{-4}$ (0.77) & $4.598\times10^{-4}$ (0.58) & $9.939\times10^{-3}$ (0.99) \\
\hline\hline
$n$ & $r$ & $\phi$ & $m_p$ & $m_\theta$ \\
\hline
16  & $1.569\times10^{-2}$ (~~--~~) & $1.028\times10^{-2}$ (~~--~~) & $1.569\times10^{-2}$ (~~--~~) & $1.028\times10^{-2}$ (~~--~~) \\
32  & $8.131\times10^{-3}$ (0.95) & $5.308\times10^{-3}$ (0.95) & $8.131\times10^{-3}$ (0.95) & $5.308\times10^{-3}$ (0.95) \\
64  & $4.115\times10^{-3}$ (0.98) & $2.673\times10^{-3}$ (0.99) & $4.115\times10^{-3}$ (0.98) & $2.673\times10^{-3}$ (0.99) \\
128 & $2.048\times10^{-3}$ (1.01) & $1.321\times10^{-3}$ (1.02) & $2.048\times10^{-3}$ (1.01) & $1.321\times10^{-3}$ (1.02) \\
\hline
\end{tabular}
\end{table}

To check whether the suboptimal scalar rates in Table~\ref{tab:3field-mms-dt} are caused by the fixed spatial error, we repeat the temporal study using the one-order-higher Taylor--Hood-type triple $[\mathbb P_3]^2\times\mathbb P_2\times\mathbb P_2$.  The same test with $N=128$ was not practical with the direct solver, so Table~\ref{tab:3field-mms-dt-high-order} reports the $N=64$ computation.  With the higher-order spatial spaces, the scalar variables also recover clean first-order temporal convergence.

\begin{table}[htbp]
\centering
\scriptsize
\caption{Temporal refinement for the boundary-bubble manufactured optimality system with fixed $N=64$ and $[\mathbb P_3]^2\times\mathbb P_2\times\mathbb P_2$ elements. Entries are relative errors; numbers in parentheses are rates from the previous time step.}
\label{tab:3field-mms-dt-high-order}
\begin{tabular}{c|cccc}
\hline
$n$ & $\boldsymbol u$ & $p$ & $\theta$ & $\boldsymbol w$ \\
\hline
16  & $1.325\times10^{-2}$ (~~--~~) & $3.876\times10^{-3}$ (~~--~~) & $2.002\times10^{-3}$ (~~--~~) & $7.746\times10^{-2}$ (~~--~~) \\
32  & $6.448\times10^{-3}$ (1.04) & $1.978\times10^{-3}$ (0.97) & $9.974\times10^{-4}$ (1.01) & $3.930\times10^{-2}$ (0.98) \\
64  & $3.179\times10^{-3}$ (1.02) & $9.998\times10^{-4}$ (0.99) & $4.981\times10^{-4}$ (1.00) & $1.980\times10^{-2}$ (0.99) \\
128 & $1.578\times10^{-3}$ (1.01) & $5.027\times10^{-4}$ (0.99) & $2.490\times10^{-4}$ (1.00) & $9.939\times10^{-3}$ (0.99) \\
\hline\hline
$n$ & $r$ & $\phi$ & $m_p$ & $m_\theta$ \\
\hline
16  & $1.578\times10^{-2}$ (~~--~~) & $1.037\times10^{-2}$ (~~--~~) & $1.578\times10^{-2}$ (~~--~~) & $1.037\times10^{-2}$ (~~--~~) \\
32  & $8.216\times10^{-3}$ (0.94) & $5.398\times10^{-3}$ (0.94) & $8.216\times10^{-3}$ (0.94) & $5.398\times10^{-3}$ (0.94) \\
64  & $4.195\times10^{-3}$ (0.97) & $2.756\times10^{-3}$ (0.97) & $4.195\times10^{-3}$ (0.97) & $2.756\times10^{-3}$ (0.97) \\
128 & $2.120\times10^{-3}$ (0.99) & $1.393\times10^{-3}$ (0.99) & $2.120\times10^{-3}$ (0.99) & $1.393\times10^{-3}$ (0.99) \\
\hline
\end{tabular}
\end{table}

\subsection*{Experiment 2: Degenerate storage matrix test}
Experiment 2 uses exactly the same manufactured solution fields
\[
(\boldsymbol u^*,p^*,\theta^*,\boldsymbol w^*,r^*,\phi^*,m_p^*,m_\theta^*)
\]
as Experiment 1.  The only change is the storage matrix $\boldsymbol S$ appearing in the scalar residual controls $m_p^*,m_\theta^*$ and in the adjoint residuals $R_p^*,R_\theta^*$.  Thus, for each storage choice below, the controls and tracking targets are regenerated by the residual formulas above using that storage matrix.  This is important: the same closed-form state is imposed, while the scalar adjoints, controls, and targets are storage-dependent so that the continuous optimality system remains exactly satisfied.

To illustrate the storage-degenerate regime covered by the effective-storage framework, we repeat the manufactured convergence test for three storage matrices:
\[
\boldsymbol S_{\rm SPD}=\begin{pmatrix}1&0.2\\0.2&1\end{pmatrix},\qquad
\boldsymbol S_{s_{pp}=0}=\begin{pmatrix}0&0\\0&1\end{pmatrix},\qquad
\boldsymbol S_{\rm rank}=c^\perp(c^\perp)^\top.
\]
The corresponding effective-storage constants $(c^\perp)^\top\boldsymbol S c^\perp$ are $1.6$, $1.0$, and $4.0$, respectively.  We use the same continuous manufactured targets and Gaussian time quadrature as in Table~\ref{tab:3field-mms-conv}, with $\Delta t=1/4096$ and meshes up to $64\times64$.  Table~\ref{tab:3field-storage-degenerate} reports the errors on the finest mesh, with the rate from the $32\times32$ mesh in parentheses.  The storage-degenerate cases produce essentially the same errors and rates as the positive definite case.

\begin{table}[htbp]
\centering
\scriptsize
\caption{Finest-mesh errors for positive definite and storage-degenerate matrices with $N=64$, $\Delta t=1/4096$, continuous manufactured targets, and Gaussian time quadrature. Entries are relative errors; numbers in parentheses are rates from $N=32$ to $N=64$.}
\label{tab:3field-storage-degenerate}
\begin{tabular}{cc|cccc}
\hline
storage & $(c^\perp)^\top\boldsymbol S c^\perp$ & $\boldsymbol u$ & $p$ & $\theta$ & $\boldsymbol w$ \\
\hline
SPD & $1.6$ & $7.221\times10^{-4}$ (1.93) & $1.085\times10^{-3}$ (1.98) & $1.098\times10^{-3}$ (1.99) & $9.478\times10^{-4}$ (1.58) \\
$s_{pp}=0$ & $1.0$ & $7.222\times10^{-4}$ (1.93) & $1.104\times10^{-3}$ (1.98) & $1.101\times10^{-3}$ (1.99) & $9.478\times10^{-4}$ (1.58) \\
rank-one & $4.0$ & $7.221\times10^{-4}$ (1.93) & $1.037\times10^{-3}$ (1.99) & $1.113\times10^{-3}$ (1.99) & $9.479\times10^{-4}$ (1.58) \\
\hline\hline
storage & $(c^\perp)^\top\boldsymbol S c^\perp$ & $r$ & $\phi$ & $m_p$ & $m_\theta$ \\
\hline
SPD & $1.6$ & $8.695\times10^{-4}$ (2.09) & $8.911\times10^{-4}$ (2.11) & $8.695\times10^{-4}$ (2.09) & $8.911\times10^{-4}$ (2.11) \\
$s_{pp}=0$ & $1.0$ & $9.002\times10^{-4}$ (2.12) & $8.933\times10^{-4}$ (2.11) & $9.002\times10^{-4}$ (2.12) & $8.933\times10^{-4}$ (2.11) \\
rank-one & $4.0$ & $9.306\times10^{-4}$ (2.14) & $9.062\times10^{-4}$ (2.12) & $9.306\times10^{-4}$ (2.14) & $9.062\times10^{-4}$ (2.12) \\
\hline
\end{tabular}
\end{table}

The only visibly lower finest-level rate in Table~\ref{tab:3field-storage-degenerate} is the rate for the adjoint displacement $\boldsymbol w$.  This reduction is not caused by storage degeneracy: the same $\boldsymbol w$ rate, about $1.58$, occurs for the positive definite storage matrix and in the baseline manufactured test.  The preceding refinement step $N=16$ to $N=32$ gives approximately second-order convergence for $\boldsymbol w$ in all three storage cases, so the final-rate reduction is interpreted as residual time-discretization sensitivity of the adjoint displacement at the finest mesh rather than a loss of stability from the degenerate storage matrix.



\subsection*{Experiment 3: Active box constraints}
To illustrate the control restriction visually with less monotone active-set geometry, we use a visualization-only trigonometric modulation of the unconstrained controls,
\[
    m_{p,\mathrm{unc}}^{\rm vis}
    = -\gamma_p^{-1}r_h\left(0.45+1.35\sin(2\pi x)\cos(2\pi y)\right),
\]
\[
    m_{\theta,\mathrm{unc}}^{\rm vis}
    = -\gamma_\theta^{-1}\phi_h\left(0.45+1.35\cos(2\pi x)\sin(2\pi y)\right).
\]
These controls are used only for the active-set visualization; the manufactured convergence tables above use the unmodified manufactured controls.  We impose the box constraints
\[
    -0.6\le m_p\le 0.6,\qquad -1\le m_\theta\le 1,
\]
and project the modulated unconstrained controls pointwise onto these intervals.  On a $32\times32$ mesh with $128$ time steps, the lower active fractions for $(m_p,m_\theta)$ are approximately $(28.47\%,30.85\%)$ at $t=0$, $(15.24\%,20.02\%)$ at $t=0.25$, and $(8.26\%,10.47\%)$ at $t=0.5$.  The corresponding upper active fractions are approximately $(6.52\%,6.89\%)$, $(4.59\%,4.96\%)$, and $(0\%,0.18\%)$.  Figure~\ref{fig:3field-active-controls} shows the more structured lower and upper active regions produced by the sign-changing trigonometric factors.  The constrained diagnostic errors are not intended as manufactured solutions convergence data; the purpose of this experiment is to visualize the pointwise projection mechanism in the variationally discretized control.

\begin{figure}[htbp]
\centering
\includegraphics[width=0.92\textwidth]{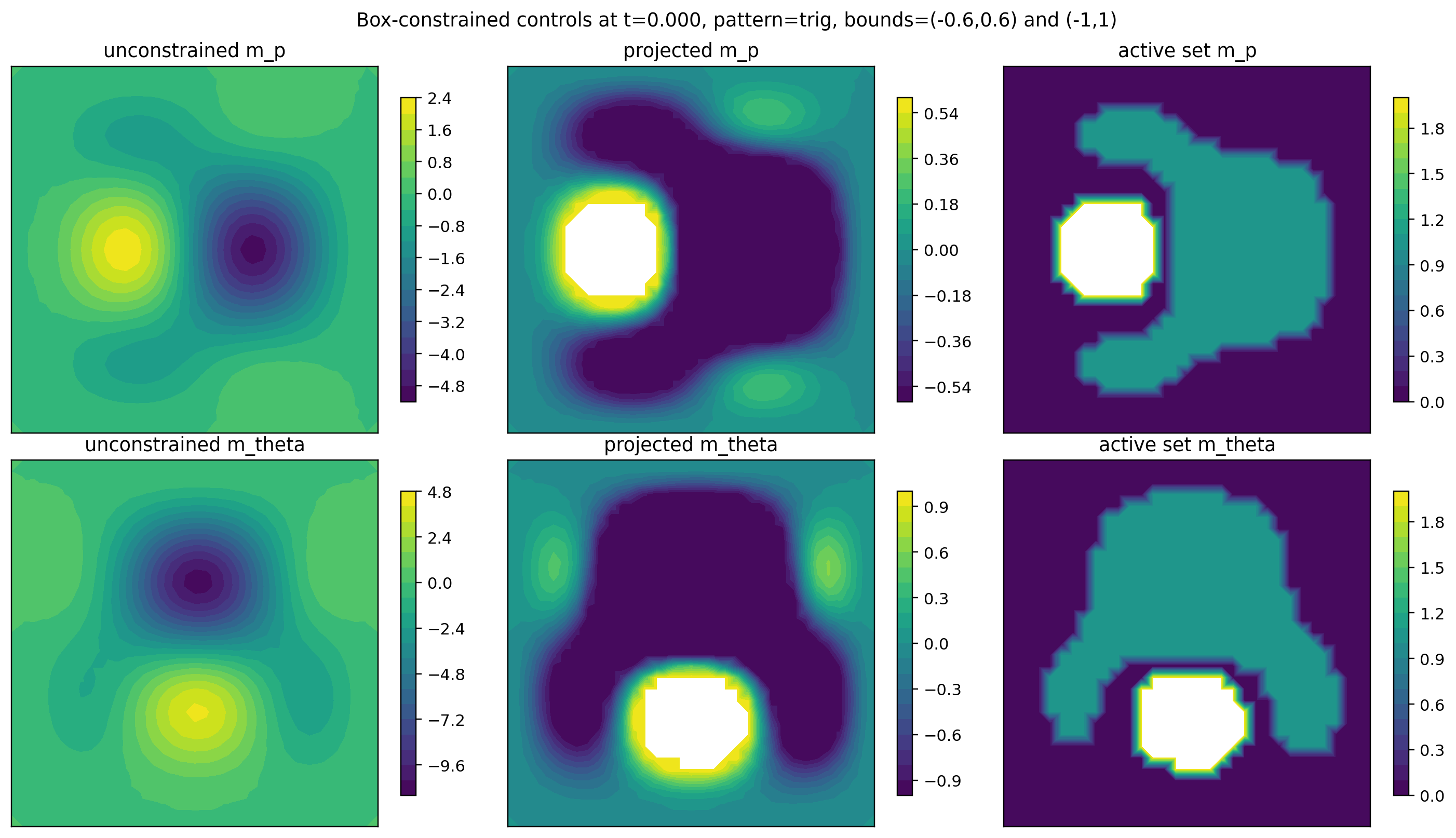}\\[1mm]
\includegraphics[width=0.92\textwidth]{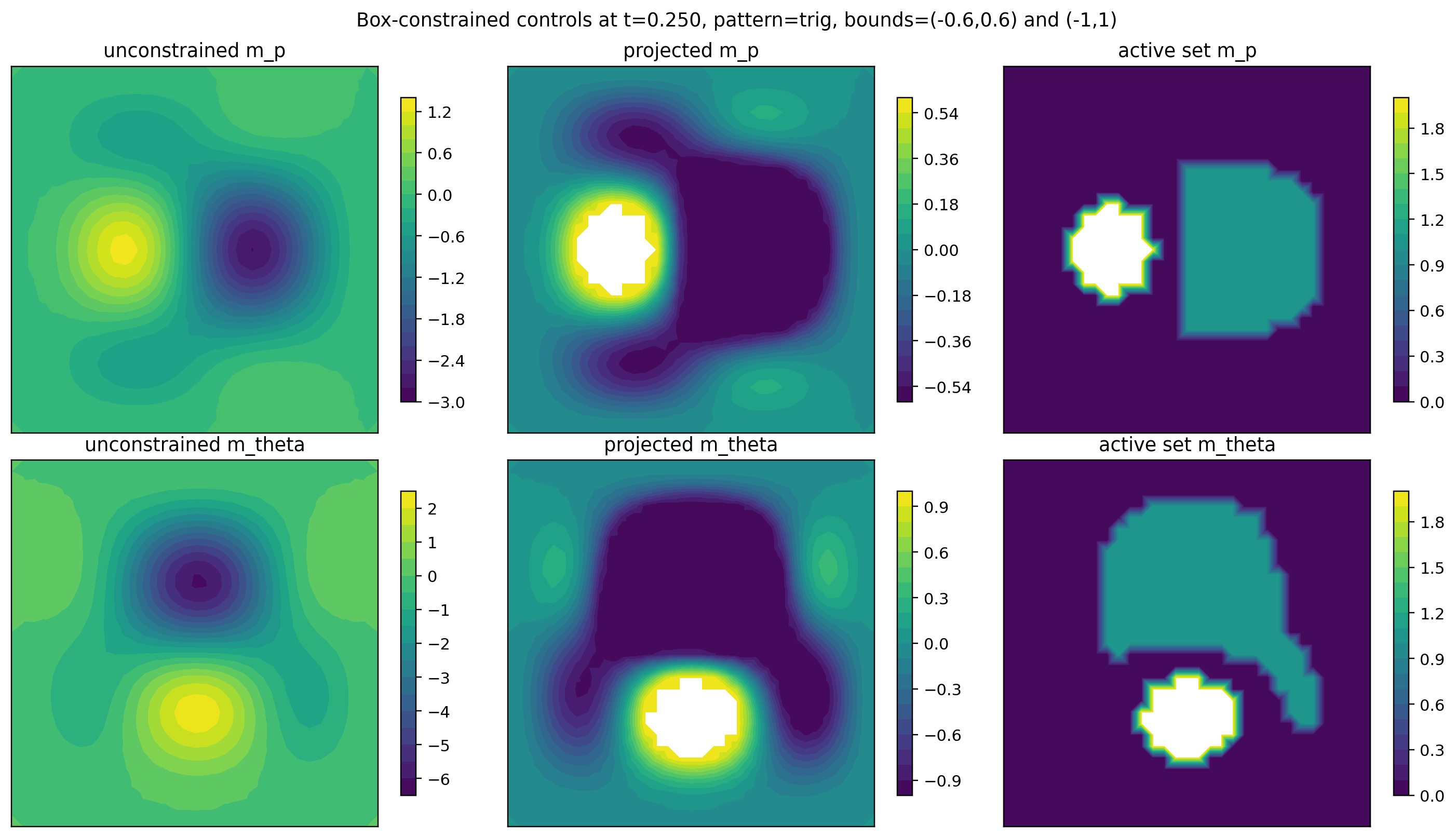}\\[1mm]
\includegraphics[width=0.92\textwidth]{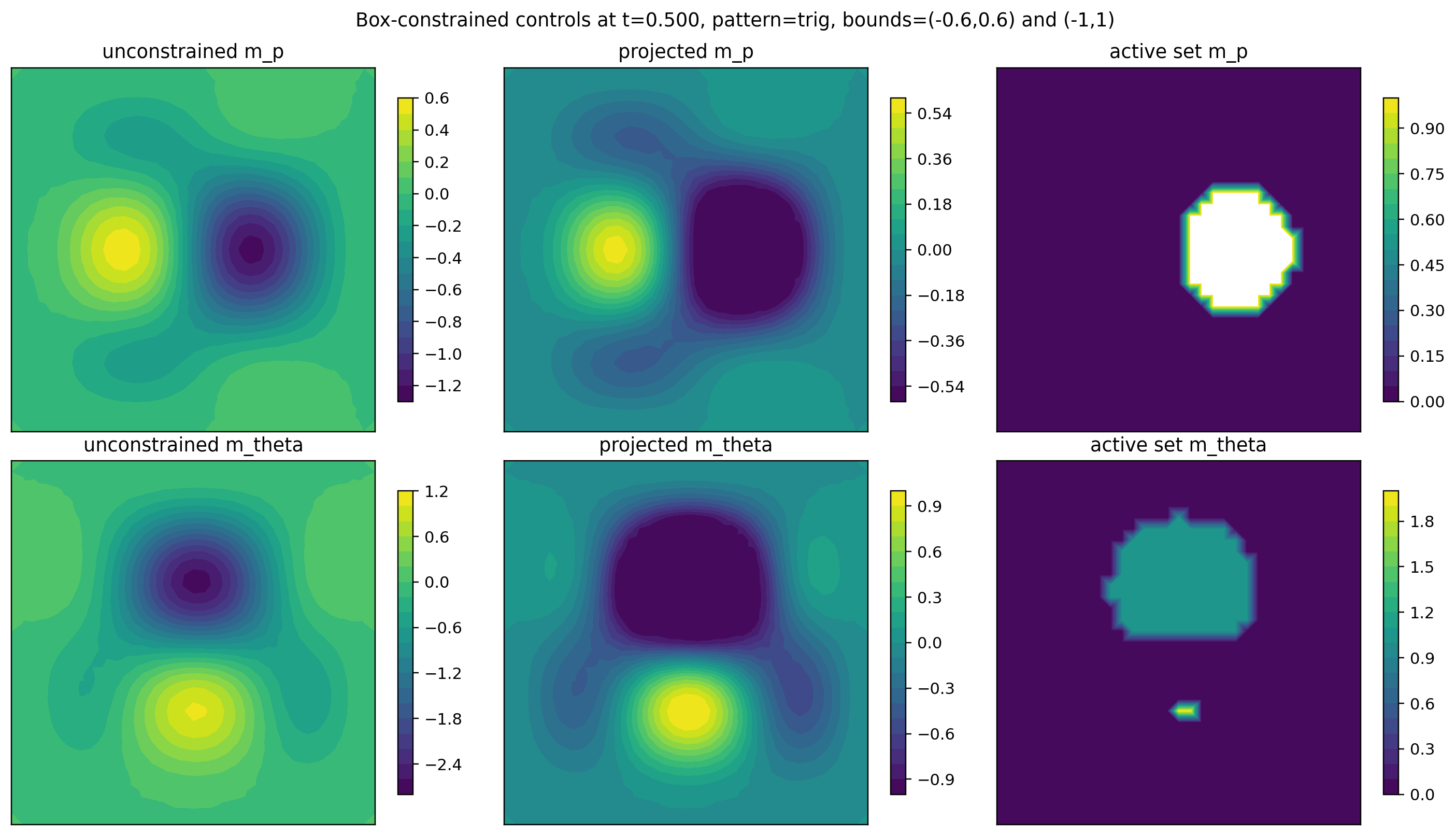}
\caption{Active box-constraint visualization on a $32\times32$ mesh with $128$ time steps at $t=0$, $t=0.25$, and $t=0.5$.  The panels show trigonometric unconstrained controls, projected controls, and active-set indicators for $m_p$ and $m_\theta$ under the bounds $-0.6\le m_p\le0.6$ and $-1\le m_\theta\le1$.}
\label{fig:3field-active-controls}
\end{figure}

\section{Conclusions}\label{sec:conclusions}
In this work, we introduced and analyzed a three-field symmetric differentiated formulation for distributed optimal control problems governed by linear thermo-poroelasticity. The state variables are the displacement $\boldsymbol{u}$, the fluid pressure $p$, and the temperature $\theta$, and the controls are the distributed fluid source $m_p$ and heat source $m_\theta$. We established well-posedness of the three-field state system, existence and uniqueness of an optimal control pair $(\bar{m}_p,\bar{m}_\theta)$, and derived the associated three-field adjoint optimality system. For a dG(0) discretization with variational discretization of both controls, we proved a priori error estimates of order $\mathcal{O}(h^s+\Delta t)$, under the stated regularity assumptions. The manufactured-solution tests in Section~\ref{Numerical Experiments} are consistent with the predicted first-order-in-time behavior of the dG(0) discretization and show higher-than-guaranteed spatial rates for the smooth manufactured solution. The storage-degenerate tests further support the effective storage framework by showing comparable convergence behavior when $\boldsymbol{S}$ is positive semidefinite rather than positive definite.

\section*{Acknowledgements}
The author used AI-based tools for language editing. The author assumes responsibility for all content.

\bibliographystyle{amsplain}
\providecommand{\bysame}{\leavevmode\hbox to3em{\hrulefill}\thinspace}
\providecommand{\MR}{\relax\ifhmode\unskip\space\fi MR }
\providecommand{\MRhref}[2]{%
  \href{http://www.ams.org/mathscinet-getitem?mr=#1}{#2}
}

\bibliography{references}

@article{biot1941general,
  AUTHOR = {Biot, Maurice A.},
  TITLE = {General theory of three-dimensional consolidation},
  JOURNAL = {J. Appl. Phys.},
  VOLUME = {12},
  YEAR = {1941},
  NUMBER = {2},
  PAGES = {155--164},
  MRNUMBER = {0004868},
  DOI = {10.1063/1.1712886},
  URL = {https://doi.org/10.1063/1.1712886}
}

@article{MR4410836,
  AUTHOR = {Bociu, Lorena and Strikwerda, Sarah},
  TITLE = {Optimal control in poroelasticity},
  JOURNAL = {Appl. Anal.},
  VOLUME = {101},
  YEAR = {2022},
  NUMBER = {5},
  PAGES = {1774--1796},
  MRNUMBER = {4410836},
  DOI = {10.1080/00036811.2021.2008372},
  URL = {https://doi.org/10.1080/00036811.2021.2008372}
}

@article{MR3504993,
  AUTHOR = {Boffi, Daniele and Botti, Michele and Di Pietro, Daniele A.},
  TITLE = {A nonconforming high-order method for the {B}iot problem on general meshes},
  JOURNAL = {SIAM J. Sci. Comput.},
  VOLUME = {38},
  YEAR = {2016},
  NUMBER = {3},
  PAGES = {A1508--A1537},
  MRNUMBER = {3504993},
  DOI = {10.1137/15M1025505},
  URL = {https://doi.org/10.1137/15M1025505}
}

@article{MR4659441,
  AUTHOR = {Cesmelioglu, Aycil and Lee, Jeonghun J. and Rhebergen, Sander},
  TITLE = {Analysis of an embedded-hybridizable discontinuous {G}alerkin method for {B}iot's consolidation model},
  JOURNAL = {J. Sci. Comput.},
  VOLUME = {97},
  YEAR = {2023},
  NUMBER = {3},
  PAGES = {Paper No. 60, 26},
  MRNUMBER = {4659441},
  DOI = {10.1007/s10915-023-02373-5},
  URL = {https://doi.org/10.1007/s10915-023-02373-5}
}

@article{MR3047799,
  AUTHOR = {Chen, Yumei and Luo, Yan and Feng, Minfu},
  TITLE = {Analysis of a discontinuous {G}alerkin method for the {B}iot's consolidation problem},
  JOURNAL = {Appl. Math. Comput.},
  VOLUME = {219},
  YEAR = {2013},
  NUMBER = {17},
  PAGES = {9043--9056},
  MRNUMBER = {3047799},
  DOI = {10.1016/j.amc.2013.03.104},
  URL = {https://doi.org/10.1016/j.amc.2013.03.104}
}

@article{MR3907413,
  AUTHOR = {Fu, Guosheng},
  TITLE = {A high-order {HDG} method for the {B}iot's consolidation model},
  JOURNAL = {Comput. Math. Appl.},
  VOLUME = {77},
  YEAR = {2019},
  NUMBER = {1},
  PAGES = {237--252},
  MRNUMBER = {3907413},
  DOI = {10.1016/j.camwa.2018.09.029},
  URL = {https://doi.org/10.1016/j.camwa.2018.09.029}
}

@article{MR2122182,
  AUTHOR = {Hinze, Michael},
  TITLE = {A variational discretization concept in control constrained optimization: the linear-quadratic case},
  JOURNAL = {Comput. Optim. Appl.},
  VOLUME = {30},
  YEAR = {2005},
  NUMBER = {1},
  PAGES = {45--61},
  MRNUMBER = {2122182},
  DOI = {10.1007/s10589-005-4559-5},
  URL = {https://doi.org/10.1007/s10589-005-4559-5}
}

@article{MR4405491,
  AUTHOR = {Khan, Arbaz and Zanotti, Pietro},
  TITLE = {A nonsymmetric approach and a quasi-optimal and robust discretization for the {B}iot's model},
  JOURNAL = {Math. Comp.},
  VOLUME = {91},
  YEAR = {2022},
  NUMBER = {335},
  PAGES = {1143--1170},
  MRNUMBER = {4405491},
  DOI = {10.1090/mcom/3699},
  URL = {https://doi.org/10.1090/mcom/3699}
}

@article{Kim1999549,
  AUTHOR = {Kim, Jun-Mo and Parizek, Richard R.},
  TITLE = {Three-dimensional finite element modelling for consolidation due to groundwater withdrawal in a desaturating anisotropic aquifer system},
  JOURNAL = {Int. J. Numer. Anal. Methods Geomech.},
  VOLUME = {23},
  YEAR = {1999},
  NUMBER = {6},
  PAGES = {549--571},
  MRNUMBER = {1696016},
  DOI = {10.1002/(SICI)1096-9853(199905)23:6<549::AID-NAG983>3.0.CO;2-Y},
  URL = {https://doi.org/10.1002/(SICI)1096-9853(199905)23:6<549::AID-NAG983>3.0.CO;2-Y}
}

@article{MR2177147,
  AUTHOR = {Korsawe, Johannes and Starke, Gerhard},
  TITLE = {A least-squares mixed finite element method for {B}iot's consolidation problem in porous media},
  JOURNAL = {SIAM J. Numer. Anal.},
  VOLUME = {43},
  YEAR = {2005},
  NUMBER = {1},
  PAGES = {318--339},
  MRNUMBER = {2177147},
  DOI = {10.1137/S0036142903432929},
  URL = {https://doi.org/10.1137/S0036142903432929}
}

@article{MR3803860,
  AUTHOR = {Lee, Jeonghun J.},
  TITLE = {Robust three-field finite element methods for {B}iot's consolidation model in poroelasticity},
  JOURNAL = {BIT},
  VOLUME = {58},
  YEAR = {2018},
  NUMBER = {2},
  PAGES = {347--372},
  MRNUMBER = {3803860},
  DOI = {10.1007/s10543-017-0688-3},
  URL = {https://doi.org/10.1007/s10543-017-0688-3}
}

@article{MR4636155,
  AUTHOR = {Liang, Hao and Rui, Hongxing},
  TITLE = {The nonconforming locking-free virtual element method for the {B}iot's consolidation model in poroelasticity},
  JOURNAL = {Comput. Math. Appl.},
  VOLUME = {148},
  YEAR = {2023},
  PAGES = {269--281},
  MRNUMBER = {4636155},
  DOI = {10.1016/j.camwa.2023.08.012},
  URL = {https://doi.org/10.1016/j.camwa.2023.08.012}
}

@incollection{malandrino2019poroelasticity,
  AUTHOR = {Malandrino, Andrea and Moeendarbary, Emad},
  TITLE = {Poroelasticity of living tissues},
  BOOKTITLE = {Encyclopedia of Biomedical Engineering},
  PUBLISHER = {Elsevier},
  YEAR = {2019},
  PAGES = {238--245}
}

@article{mccormack2020modeling,
author = {McCormack, Kimberly and Hesse, Marc A. and Dixon, Timothy and Malservisi, Rocco},
title = {Modeling the Contribution of Poroelastic Deformation to Postseismic Geodetic Signals},
journal = {Geophysical Research Letters},
volume = {47},
number = {8},
pages = {e2020GL086945},
doi = {https://doi.org/10.1029/2020GL086945},
url = {https://agupubs.onlinelibrary.wiley.com/doi/abs/10.1029/2020GL086945},
eprint = {https://agupubs.onlinelibrary.wiley.com/doi/pdf/10.1029/2020GL086945},
note = {e2020GL086945 10.1029/2020GL086945},
year = {2020}
}

@article{phillips2008coupling,
  AUTHOR = {Phillips, Phillip Joseph and Wheeler, Mary F.},
  TITLE = {A coupling of mixed and discontinuous {G}alerkin finite-element methods for poroelasticity},
  JOURNAL = {Comput. Geosci.},
  VOLUME = {12},
  YEAR = {2008},
  NUMBER = {4},
  PAGES = {417--435},
  MRNUMBER = {2454687},
  DOI = {10.1007/s10596-008-9082-1},
  URL = {https://doi.org/10.1007/s10596-008-9082-1}
}

@article{MR3606362,
  AUTHOR = {Rivi{\`e}re, B{\'e}atrice and Tan, Jun and Thompson, Travis},
  TITLE = {Error analysis of primal discontinuous {G}alerkin methods for a mixed formulation of the {B}iot equations},
  JOURNAL = {Comput. Math. Appl.},
  VOLUME = {73},
  YEAR = {2017},
  NUMBER = {4},
  PAGES = {666--683},
  MRNUMBER = {3606362},
  DOI = {10.1016/j.camwa.2016.12.030},
  URL = {https://doi.org/10.1016/j.camwa.2016.12.030}
}

@article{Swan200325,
  AUTHOR = {Swan, Colby C. and Lakes, R. S. and Brand, R. A. and Stewart, K. J.},
  TITLE = {Micromechanically based poroelastic modeling of fluid flow in haversian bone},
  JOURNAL = {J. Biomech. Eng.},
  VOLUME = {125},
  YEAR = {2003},
  NUMBER = {1},
  PAGES = {25--37}
}

@article{MR4221326,
  AUTHOR = {Tang, Xialan and Liu, Zhibin and Zhang, Baiju and Feng, Minfu},
  TITLE = {On the locking-free three-field virtual element methods for {B}iot's consolidation model in poroelasticity},
  JOURNAL = {ESAIM Math. Model. Numer. Anal.},
  VOLUME = {55},
  YEAR = {2021},
  PAGES = {S909--S939},
  MRNUMBER = {4221326},
  DOI = {10.1051/m2an/2020064},
  URL = {https://doi.org/10.1051/m2an/2020064}
}

@article{WANGEN2016486,
  AUTHOR = {Wangen, Magnus and Gasda, Sarah and Bj{\o}rnar{\aa}, Tore},
  TITLE = {Geomechanical consequences of large-scale fluid storage in the utsira formation in the north sea},
  JOURNAL = {Energy Procedia},
  VOLUME = {97},
  YEAR = {2016},
  PAGES = {486--493}
}

@article{MR2273503,
  AUTHOR = {Weinstein, Tessa and Bennethum, Lynn S.},
  TITLE = {On the derivation of the transport equation for swelling porous materials with finite deformation},
  JOURNAL = {Int. J. Eng. Sci.},
  VOLUME = {44},
  YEAR = {2006},
  NUMBER = {18-19},
  PAGES = {1408--1422},
  MRNUMBER = {2273503},
  DOI = {10.1016/j.ijengsci.2006.08.001},
  URL = {https://doi.org/10.1016/j.ijengsci.2006.08.001}
}

@article{MR2644299,
  AUTHOR = {Zhou, Zhaojie and Yan, Ningning},
  TITLE = {The local discontinuous {G}alerkin method for optimal control problem governed by convection diffusion equations},
  JOURNAL = {Int. J. Numer. Anal. Model.},
  VOLUME = {7},
  YEAR = {2010},
  NUMBER = {4},
  PAGES = {681--699},
  MRNUMBER = {2644299}
}

@article {ZhangRui2022,
    AUTHOR = {Zhang, Jing and Rui, Hongxing},
     TITLE = {Galerkin method for the fully coupled quasi-static
              thermo-poroelastic problem},
   JOURNAL = {Comput. Math. Appl.},
  FJOURNAL = {Computers \& Mathematics with Applications. An International
              Journal},
    VOLUME = {118},
      YEAR = {2022},
     PAGES = {95--109},
      ISSN = {0898-1221,1873-7668},
   MRCLASS = {65N30 (74F10 80A19)},
  MRNUMBER = {4432106},
MRREVIEWER = {Hao\ Dong},
       DOI = {10.1016/j.camwa.2022.04.019},
       URL = {https://doi.org/10.1016/j.camwa.2022.04.019},
}

@article{YiLee2024,
  AUTHOR = {Yi, Son-Young and Lee, Sanghyun},
  TITLE = {Physics-preserving enriched {G}alerkin method for a fully-coupled thermo-poroelasticity model},
  JOURNAL = {Numer. Math.},
  VOLUME = {156},
  YEAR = {2024},
  NUMBER = {3},
  PAGES = {949--978},
  MRNUMBER = {4836466},
  DOI = {10.1007/s00211-024-01406-x},
  URL = {https://doi.org/10.1007/s00211-024-01406-x}
}

@article {AntoniettiBonettiBotti2023,
    AUTHOR = {Antonietti, Paola F. and Bonetti, Stefano and Botti, Michele},
     TITLE = {Discontinuous {G}alerkin approximation of the fully coupled
              thermo-poroelastic problem},
   JOURNAL = {SIAM J. Sci. Comput.},
  FJOURNAL = {SIAM Journal on Scientific Computing},
    VOLUME = {45},
      YEAR = {2023},
    NUMBER = {2},
     PAGES = {A621--A645},
      ISSN = {1064-8275,1095-7197},
   MRCLASS = {65M60 (65M12 74F05 76S05)},
  MRNUMBER = {4579739},
       DOI = {10.1137/22M1498747},
       URL = {https://doi.org/10.1137/22M1498747},
}

@article {ChenCuiZhou2026,
    AUTHOR = {Chen, Fan and Cui, Ming and Zhou, Chenguang},
     TITLE = {Sequential symmetric interior penalty discontinuous {G}alerkin
              method for fully coupled quasi-static thermo-poroelasticity
              problems},
   JOURNAL = {Comput. Math. Appl.},
  FJOURNAL = {Computers \& Mathematics with Applications. An International
              Journal},
    VOLUME = {218},
      YEAR = {2026},
     PAGES = {67--88},
      ISSN = {0898-1221,1873-7668},
   MRCLASS = {65M15 (65M60 74F05 74F10 76M10)},
  MRNUMBER = {5088347},
       DOI = {10.1016/j.camwa.2026.06.012},
       URL = {https://doi.org/10.1016/j.camwa.2026.06.012},
}

@article{KhanLeeSingh2026,
  AUTHOR = {Khan, Arbaz and Lee, Jeonghun J. and Singh, Harpal},
  TITLE = {Distributed optimal control problems governed by poroelasticity equations},
  JOURNAL = {arXiv preprint arXiv:2605.30839},
  YEAR = {2026},
  EPRINT = {2605.30839},
  ARCHIVEPREFIX = {arXiv},
  PRIMARYCLASS = {math.OC}
}

@article {MR4146798,
    AUTHOR = {Brun, Mats Kirkes\ae ther and Ahmed, Elyes and Berre, Inga and
              Nordbotten, Jan Martin and Radu, Florin Adrian},
     TITLE = {Monolithic and splitting solution schemes for fully coupled
              quasi-static thermo-poroelasticity with nonlinear convective
              transport},
   JOURNAL = {Comput. Math. Appl.},
  FJOURNAL = {Computers \& Mathematics with Applications. An International
              Journal},
    VOLUME = {80},
      YEAR = {2020},
    NUMBER = {8},
     PAGES = {1964--1984},
      ISSN = {0898-1221,1873-7668},
   MRCLASS = {65M60 (65M22 74F10 76S05)},
  MRNUMBER = {4146798},
       DOI = {10.1016/j.camwa.2020.08.022},
       URL = {https://doi.org/10.1016/j.camwa.2020.08.022},
}
\end{document}